\documentclass[a4paper,reqno,11pt]{amsart}

\usepackage{etoolbox}
\patchcmd{\section}{\scshape}{\bfseries\large}{}{}
\patchcmd{\subsubsection}{\itshape}{}{}{}
\makeatletter
\def\@seccntformat#1{\csname the#1\endcsname.\space}
\makeatother

\usepackage[a4paper,scale={0.72,0.74},marginratio={1:1},footskip=7mm,headsep=10mm]{geometry}

\usepackage[utf8]{inputenc}
\usepackage[T1]{fontenc}

\usepackage{amsmath, amsfonts, amssymb, amsthm, amscd}
\usepackage{graphicx}
\usepackage{psfrag}
\usepackage{perpage}
\usepackage{url}
\usepackage{color}
\usepackage{mathrsfs}
\usepackage{mathabx}
\usepackage{scalerel}
\usepackage{enumerate}
\usepackage{comment}
\usepackage[normalem]{ulem}
\usepackage{tikz}
\usepackage{caption,subcaption}
\usetikzlibrary{arrows,decorations.pathmorphing,backgrounds,positioning,fit,petri}
\usepackage{tikz-cd}
\usepackage{subdepth}

\usepackage{dsfont} % nice indicator function symbol (blackboard 1)
\usepackage{microtype}
\usepackage{enumitem}
\usepackage{hyperref}
\allowdisplaybreaks[0] % disable align page breaks

\numberwithin{equation}{section}

\newtheorem{theorem}{Theorem}[section]
\newtheorem{lemma}[theorem]{Lemma}
\newtheorem{proposition}[theorem]{Proposition}
\newtheorem{corollary}[theorem]{Corollary}

\theoremstyle{definition}
\newtheorem{definition}[theorem]{Definition}

\newtheorem{remark}[theorem]{Remark}

\newcommand{\E}{\mathbb{E}}
\newcommand{\N}{\mathbb{N}}
\renewcommand{\P}{\mathbb{P}}

\newcommand{\R}{\mathbb{R}}
\newcommand{\Z}{\mathbb{Z}}
\newcommand{\T}{\mathbb{T}}

\def\bs{\boldsymbol}

\newcommand\bP{\ensuremath{\bs{\mathrm{P}}}}
\newcommand\bE{\ensuremath{\bs{\mathrm{E}}}}

\newcommand{\cC}{{\ensuremath{\mathcal C}} }

\newcommand{\cT}{{\ensuremath{\mathcal T}} }

\newcommand{\weight}{\textup{\textrm{w}}}

\renewcommand{\theta}{\vartheta}
\renewcommand{\rho}{\varrho}

\newcommand{\effR}[3]{{\rm R}_{\rm eff}^{#1}(#2 \leftrightarrow #3)} % Effective resistance

\newcommand{\limn}{\lim\limits_{n \rightarrow \infty}}

\makeatletter
\newcommand\xleftrightarrow[2][]{%
  \ext@arrow 9999{\longleftrightarrowfill@}{#1}{#2}}
\newcommand\longleftrightarrowfill@{%
  \arrowfill@\leftarrow\relbar\rightarrow}
\makeatother

\MakePerPage[2]{footnote} % restarts footnote counter at every new page

\title[Local Limit of RSTRE]{Local limits of random spanning trees in random environment}
\date{\today}

\author[L. Makowiec]{Luca Makowiec}
\address{Department of Mathematics\\
National University of Singapore\\
10 Lower Kent Ridge Road, 119076 Singapore
}

\email{l.m.makowiec@gmail.com}

\keywords{disordered system,  minimum spanning tree, random graphs, uniform spanning tree}
\subjclass[2020]{Primary: 60K35;  Secondary: 82B41, 82B44, 05C05}

\begin{document}

\begin{abstract}
 We study the edge overlap and local limit of the random spanning tree in random environment (RSTRE) on the complete graph with $n$ vertices and weights given by $\exp(-\beta \omega_e)$ for $\omega_e$ uniformly distributed on $[0,1]$. We show that for $\beta$ growing with $\beta = o(n/\log n)$, the edge overlap is $(1+o(1)) \beta$, while for $\beta$ much larger than $n \log^2 n$, the edge overlap is $(1-o(1))n$. Furthermore, there is a transition of the local limit around $\beta = n$. When $\beta = o(n/ \log n)$ the RSTRE locally converges to the same limit as the uniform spanning tree, whereas for $\beta$ larger than $n \log^\lambda n$, where $\lambda = \lambda(n) \rightarrow \infty$ arbitrarily slowly, the local limit of the RSTRE is the same as that of the minimum spanning tree.
\end{abstract}

\maketitle
%%%%%%%%%%%%%%%%%%%%%%%%%%%%%%%%%%%%%%%%%%%%%%
%%%% Main text entry area:

\section{Introduction}

Let $K_n = (V_n, E_n) = (V,E)$ be the complete graph on the vertex set $V = \{1, \ldots, n\}$, and let $\omega = (\omega_e)_{e \in E}$ be i.i.d.\ random variables uniformly distributed on $[0, 1]$, defined on some probability space with underlying probability measure $\P$. The \textit{random spanning tree in random environment} (RSTRE) on $K_n$ is defined as follows.  Given a realization of the random environment $(\omega_e)_{e \in E}$ and $\beta \geq 0$, we let the weight (or conductance) of an edge be $\weight(e) = \exp(-\beta \omega_e)$. We then define the (weighted) uniform spanning tree measure on $\T(K_n)$, the set of spanning trees on $K_n$, where we identify each tree with its set of edges, as
\begin{equation}\label{eq:PomegaT}
	\bP_{n, \beta}^{\omega}(\mathcal{T} = T) :=\frac{1}{Z_{\beta}^{\omega}} \prod_{e \in T} \weight(e) = \frac{1}{Z_{\beta}^{\omega}} \prod_{e \in T} \exp(-\beta \omega_e),
\end{equation}
where $Z_{\beta}^{\omega}$ is the normalization constant, also known as the partition function. We write $\bE^\omega_{n,\beta}$ for the expectation w.r.t.\ $\bP_{n, \beta}^{\omega}$, and we will often drop the sub- and superscripts when there is no ambiguity. Instead of studying the quenched law \eqref{eq:PomegaT} for fixed realizations of $\omega$, we will study the averaged (over $\omega$) law 
\begin{equation*}
	\widehat{\P}_{n,\beta}(\cdot) = \widehat{\P}(\cdot) = \E \big[ \bP( \cdot) \big]
\end{equation*}
with corresponding expectation $\widehat{\E}$. 

In \cite{MSS23}, the authors showed that for arbitrary distributions, but fixed $\beta$, the diameter of typical trees on sequences of either bounded degree expanders or boxes in $\Z^d$, $d\geq 5$, grows asymptotically as $|V|^{1/2 + o(1)}$. In \cite{MSS24}, the same authors further analyzed the diameter of a typical tree on the complete graph with the environment being distributed as i.i.d.\ uniforms on $[0,1]$. They show the existence of two regimes: (1) a low disorder regime $\beta \leq n^{1 - o(1)}$, where the diameter is of the same order as the diameter of the \textit{uniform spanning tree} (UST); (2) a high disorder regime $\beta \geq n^{4/3 + o(1)}$, where the diameter is comparable to the diameter of the random \textit{minimum spanning tree} (MST). In this paper, we shall study local observables for the RSTRE with the same setting as in \cite{MSS24}.

\subsection{Main results}

Fix some realization of the environment $\omega = (\omega_e)_{e \in E}$, and let $\cT$ and $\cT'$ be two spanning trees sampled independently according to $\bP^\omega_{n, \beta}$. The first observable that we shall study is the expected edge overlap
\begin{equation} \label{eq:def_overlap}
	\mathcal{O}(\beta) :=  \bE^{\omega}_{n, \beta} \, \otimes \, \bE^{\omega}_{n, \beta} \big[ |\cT \cap \cT'| \big] = \sum_{T, T' \in \T(K_n)} | T \cap T'|\, \bP^\omega_{n, \beta}(\cT = T) \bP^\omega_{n, \beta}(\cT = T'),
\end{equation}
where we identify the trees with their edge sets and let $\beta = \beta(n)$ depend on $n$. Our first main result is the following, where we say that a sequence of events $A_n$ holds with high probability if $\widehat{\P}(A_n) \rightarrow 1$ as $n$ tends to infinity (see also Section~\ref{SS:notation} for a remark about the $o_\beta(1)$ notation).

\begin{theorem} \label{T:Overlap}
	Let $\cT$ be the RSTRE on the complete graph with $n$ vertices and random environment distributed as uniform random variables on $[0,1]$. 
    \begin{enumerate}[label=\roman*)]
        \item \label{T:Overlap_low} If $0 < \beta = \beta(n) \ll n/\log n$, then with high probability
	\begin{equation} \label{eq:overlap_low}
		\mathcal{O}(\beta) = \big(1 + o_\beta(1) \big) \, \beta \frac{1 - e^{-2\beta}}{(1 -e^{-\beta})^2}.
	\end{equation}  
        \item \label{T:Overlap_high} If $\beta \gg n (\log n)^2$, then with high probability
	\begin{equation} 
		\mathcal{O}(\beta) = \big(1 - o_\beta(1) \big) \, n.
	\end{equation} 
    \end{enumerate}
\end{theorem}

\begin{remark}
	As the probability of an edge being in the (unweighted) UST on the complete graph (which corresponds to $\beta = 0$ in \eqref{eq:PomegaT}) is $2/n$ and is independent of $\omega$, the expected overlap of the UST is $2 (n-1)/n$. On the other hand, for the MST the overlap is trivially $(n-1)$. 
	%In fact, when $\beta \gg n (\log n)^2$ the RSTRE will be equal to the MST up to $o(n)$ number of edges.
\end{remark}

Our second result is the local (weak) convergence of the RSTRE to either the UST local limit or the MST local limit depending on the strength of the disorder $\beta = \beta(n)$. We refer to Section~\ref{SS:local}, and in particular to \eqref{eq:def_local_conv}, for a characterization of local (weak) convergence of the RSTRE.

\begin{theorem} \label{T:local_limit}
	Let $\cT$ be the RSTRE on the complete graph with $n$ vertices and random environment distributed as uniform random variables on $[0,1]$. 
    \begin{enumerate}[label=\roman*)]
        \item  \label{T:local_limit_low} If $\beta = \beta(n) \ll n/\log n$, then under the averaged law $\widehat{\P}$ we have the local convergence
	\begin{equation} \label{eq:localRSTRE_UST}
		(\cT, 1) \xrightarrow{d} (\mathcal{P},o),
	\end{equation}
	where $(\cT,1)$ is $\cT$ rooted at $1$ and $\mathcal{P}$ is the tree associated to a Poisson(1) branching process conditioned to survive forever. 
    
    \item \label{T:local_limit_high} If on the other hand $\beta = \beta(n) \geq n (\log n)^\lambda$, where $\lambda = \lambda(n) \rightarrow \infty$ arbitrarily slowly, then under the averaged law $\widehat{\P}$ we have the local convergence
	\begin{equation} \label{eq:localRSTRE_MST}
		(\cT, 1) \xrightarrow{d} (\mathcal{M}, o),
	\end{equation}
    where the (random) rooted tree $(\mathcal{M}, o)$ is the local limit of the MST on the complete graph.
    \end{enumerate} 
\end{theorem}

\noindent To construct $\mathcal{P}$, %the tree associated to the Poisson(1) branching process conditioned to survive forever, 
we first consider an infinite backbone $x_0, x_1, x_2, \ldots$ started at the root $x_0 = o$, and then attach trees associated to independent Poisson(1) branching processes to each vertex $x_i$. For a large class of graphs, this is the local limit of the UST (see e.g.\ \cite{BP93} and \cite{NP22}). In \cite{Add13}, the author gives an explicit description of the tree $\mathcal{M}$, by running
a process called invasion percolation from every vertex of the so-called Poisson weighted infinite tree, and taking the union of all the components containing the origin. %it is shown that the tree $\mathcal{M}$ may be constructed by starting an invasion percolation process at every vertex of the Poisson weighted infinite tree rooted at $o$, and then taking the union of all clusters containing $o$. 
We refer to Section \ref{SS:local} for more details about local convergence of the UST and MST.

Theorem~\ref{T:local_limit} shows that there is a sharp transition for the local limit when $\beta = n^\gamma$ and $\gamma$ crosses the critical value $\gamma_c = 1$. This stands in contrast to \cite[Conjecture 1.3]{MSS24}, where the diameter is conjectured to smoothly interpolate between $n^{1/2}$ and $n^{1/3}$ when $\gamma \in [1, 4/3]$. It would be interesting to study what the local limit (if it exists) should be when $\beta$ is of order $n$.

\smallskip

Very shortly after we posted this paper, \cite{K24} uploaded a paper studying the same model as ours. In particular, using the Aldous-Broder algorithm, they manage to (independently) prove Theorem~1.1 of \cite{MSS24}. Furthermore, they study yet another observable called the total length, which in our notation is defined as
\begin{equation*}
	L(\cT) := \sum_{e \in \cT} \omega_e,
\end{equation*}
and is equivalent to the Hamiltonian $H(\cT, \omega)$ as in \cite{MSS24}. In the following, we state an improvement to the low disorder part of \cite[Theorem~1.8]{K24}.
\begin{theorem} \label{T:length}
	Let $\cT$ be the RSTRE on the complete graph with $n$ vertices and random environment distributed as uniform random variables on $[0,1]$. 
    \begin{enumerate}[label=\roman*)]
        \item \label{T:length_low} If $0 < \beta = \beta(n) \ll n/\log n$, then with high probability (and in $\P$-expectation)
	\begin{equation} \label{eq:length_low}
		\bE^\omega_{n,\beta} \big[ L(\cT) \big]= \big(1 + o_\beta(1) \big) \frac{n}{\beta} \cdot \frac{1 - \beta e^{-\beta} - e^{-\beta}}{1- e^{-\beta}}.
	\end{equation} 

    \item \label{T:length_high} If $\beta \gg n (\log n)^5$, then with high probability (and in $\P$-expectation)
	\begin{equation} \label{eq:length_high}
		\bE^\omega_{n,\beta} \big[ L(\cT) \big] = \big(1 + o_\beta(1) \big) \zeta(3),
	\end{equation}
	where $\zeta(3) = \sum_{k=1}^\infty k^{-3} = 1.202\dots$ is Apéry's constant. \label{enu:length_high}
    \end{enumerate} 
\end{theorem}
The constant $\zeta(3)$ is precisely the limiting total length of the MST on the complete graph, as determined in \cite{Fri85}. Theorem~\ref{T:length} will be a consequence of the tools we develop for the edge overlap, see the end of Sections~\ref{SS:effR_up} and \ref{SS:overlap_high} for a (sketch) of the proofs. We remark that in \cite[Theorem~1.8]{K24}, it is shown that the equality \eqref{eq:length_high} in the high disorder holds under the milder condition $\beta \gg n \log n$. See also the PhD thesis \cite[Section 4.4]{Mak25} of the author for yet another alternative simple proof of this fact. Additionally, in a newer version of \cite{K24}, it is proven that the RSTRE local limit agrees with that of the MST local limit under the milder condition $\beta \gg n \log n$, giving an improvement to the case of large $\beta$ in Theorem~\ref{T:local_limit}.

%%%%%%%%%%%%%%%%%%%%%%%%%%%%%%%%%%%%%%%%%%%%%%%%%%%%%%

\subsection{Outline and proof ideas}%
We split the proofs of Theorems~\ref{T:Overlap}, \ref{T:local_limit} and \ref{T:length} into two sections: in Section \ref{S:low} we cover the theorems in the low disorder regime ($\beta \ll n/\log n$), whereas Section \ref{S:high} covers the high disorder regime ($\beta \geq n^{1+o(1)}$). Section \ref{S:prelim} covers some preliminaries about spanning trees and effective resistances. Both the proofs of Theorem~\ref{T:Overlap} and Theorem~\ref{T:local_limit} will make use of the fact that we are able to give sharp bounds on the effective resistance in the weighted graph between pairs of vertices (see Definition~\ref{D:effR}) so that Kirchhoff's formula (Theorem \ref{T:Kirchhoff}) allows us to compute the probabilities of edges being in the tree. 

Note that edges with large weight, say larger than $e^{-1}$, correspond to edges $e$ with $\omega_e < 1/\beta$. These edges form a subgraph distributed as an Erdős-R\'enyi random graph $G(n,p)$ with $p=1/\beta$. The proof techniques differ depending on the choice of $p$ (or equivalently $\beta$). When $p \gg \log n/n$, the graph $G(n,p)$ is well connected and in some sense behaves like a $pn$-regular graph (with $pn$ diverging), so that the tree $\cT$ has many edges to choose from. %When $p$ is much smallerFurthermore, notice that for two edges $e,f \in E$, the weight of $f$ is negligible compared to the weight of $e$ if $\omega_f - \omega_e \gg \log n/\beta$. 
When $p = 1/\beta \ll 1/n$, the corresponding Erdős-R\'enyi random graph is disconnected with (almost cycle--free) components of at most logarithmic size. In this case, paths consisting of large weighted edges (these paths will be in the MST) have only very little competition from other edges, and hence they influence the UST measure much more than in the low disorder case. % If $p =1/\beta \approx 1/n$, we believe that there is a subgraph $H \subseteq K_n$ with on the order of $n = o(|E_n|)$ many edges such that the effective resistance between pairs of vertices on $H$ well approximate the corresponding effective resistances in $(K_n, \weight)$.} so that the values are typically far away from their analogs in the graph with constant weights.

%\smallskip 

More precisely, for $\beta \ll n/\log n$, we will show that the effective resistance between any two endpoints of an edge in the weighted graph $(K_n, \weight)$ concentrates around the value that we get in a graph where all the weights are replaced by their expected value. The main idea is that we can approximate the weighted graph $(K_n, \weight)$ by a union of disjoint Erdős-R\'enyi random graphs of parameter $p$, with discretized and constant weights, to give bounds on the effective resistance on $(K_n,\weight)$.  We choose $p = 1/\beta$, such that whenever $\beta \ll n/\log n$ the random graph is connected, and as there will be many edges, the effective resistance (on the unweighted graph) will be of the correct order. Once $\beta \gg n/\log n$, this approximation breaks down. 
%\smallskip

In fact, when $\beta \geq n^{1+o(1)}$, we can identify most of the edges that contribute significantly to the effective resistance. These will (mostly) be the edges belonging to the random MST $M$, and with high probability many (but not all) of these edges will also be in the weighted spanning tree. This means that $| E(\cT) \setminus E(M) | = o(|E(\cT)|) = o(n)$, and thus for independent trees $\cT$ and $\cT'$, we will have $|E(\cT) \cap E(\cT')| = n - o(n)$. Furthermore, as the size of the local neighborhoods of the trees does not grow too fast, the tree will locally not observe these $o(n)$ edges that disagree with the MST.

\subsection{Notation} \label{SS:notation}
We write $(G, \weight)$ for a finite graph with positive weights $\weight(\cdot)$ on the undirected edges $e = \{u, v\}$, and we write $(G,1)$ for the case when all the weights are equal to one. We denote by $E^\rightarrow = \{ (u,v) : \{u,v\} \in E \}$ the set of directed edges, where we include each edge in both orientations. For each edge $e = (u,v) \in E^\rightarrow$ we write $e^- = u$ and $e^+ = v$ for the starting and endpoints of $e$. Additionally, for each undirected edge in $E$ we (arbitrarily) fix one orientation and write $E^\rightarrow_{\textrm{fix}} \subset E^\rightarrow$ for the set of these directed edges. The edge sets $E$ and $E^\rightarrow_{\textrm{fix}}$ are then in a one-to-one correspondence with each other.
%and it is useful to think of $E$ containing each edge in both orientations. 
We often identify graphs and subgraphs with their own edge sets. Specifically, if $G_1$ and $G_2$ are two graphs on the same vertex set $V$, we denote by $G_1 \cap G_2$ the graph consisting of the vertex set $V$ with edges $E(G_1) \cap E(G_2)$. %, that is, 

For functions $f,g:\N \rightarrow [0,\infty)$, we write either $f\ll g$ or $f=o(g)$ if $f(n)/g(n)\rightarrow 0$ as $n\rightarrow \infty$, and we write $f \gg g$ if $g(n)/f(n)\rightarrow 0$ as $n\rightarrow \infty$. Furthermore, we write $f=O(g)$ if there is a constant $C\geq 0$ such that $f(n)\leq Cg(n)$ for all large enough $n$. When either $o(g)$ or $O(g)$ depends on some parameter, which itself may also depend on $n$ (e.g. $\beta = \beta(n)$), we add a subscript and write $o_\beta(g)$ or $O_\beta(g)$. For instance, if $\beta(n) \ll n/\log n$ and we say that for a random variable $X_n(\beta)$ it holds that
\begin{equation*}
    X_n(\beta) = \big(1 + o_\beta(g_1) \big) f(\beta)
\end{equation*}
 with $\widehat{\P}$- (or $\P$-) probability at least $1- o_\beta(g_2)$, then we mean that there exists functions $h_1, h_2 : \R \times \N \rightarrow [0,\infty)$, with $h_1(\beta(n),n) = o(g_1)$ and $h_2(\beta(n), n) = o(g_2)$ whenever $\beta(n) \ll n/\log n$, such that 
\begin{equation*}
    \widehat{\P} \Big( \big| X_n(\beta) - f(\beta) \big| \geq h_1(\beta(n),n) \cdot f(\beta) \Big) \leq h_2(\beta(n), n).
\end{equation*}
Similar notions hold when we replace $o_\beta(\cdot)$ with  $O_\beta(\cdot)$. Finally, if an event occurs with probability at least $1- o_\beta(1)$ (i.e.\ $g_2 \equiv 1$), then we say that it holds with high probability.

\section{Preliminaries} \label{S:prelim}

In this section, we first recall the relevant notation for electric networks, following mainly the exposition of Chapter 2 in \cite{LP16}. Then we discuss a coupling of the random environment and the MST to an Erdős-R\'enyi random graph, and lastly, we recall the notion of local convergence of graphs needed for Theorem~\ref{T:local_limit}.

\subsection{Electric Networks and UST} \label{SS:electric}

Let $(G, \weight)$ be a graph with weights (or conductances) $(\weight(e))_{e \in E}$ on the edges. We may think of $1/\weight(e)$ as the resistance of an edge. Given disjoint vertex sets $A, B \subset V$, a unit flow $\theta : E^\rightarrow \rightarrow \R$ from $A$ to $B$ is an antisymmetric function that has net flow out of $A$ equal to $1$, net flow out of $B$ equal to $-1$, and zero net flow out of any vertices not in $A \cup B$. That is, $\theta$ satisfies
\begin{enumerate}
	\item $\theta(e^-,e^+) = - \, \theta(e^+,e^-)$;
	\item $\sum_{e^- = x} \theta(e) = 0$ for all $x \not\in A \cup B$;
	\item and 
	\begin{equation*}
		\sum_{a \in A} \sum_{e^- = a } \theta(e) = 1 = - \sum_{b \in B}\sum_{e^- = b} \theta(e). 
	\end{equation*}
\end{enumerate}
The energy of a flow $ \Vert\theta \Vert^2_r$ is defined as 
\begin{equation}
	\Vert\theta \Vert^2_r : = \frac{1}{2} \sum_{e \in E^\rightarrow} \frac{1}{\weight(e)} \theta(e)^2,
\end{equation}
where each edge appears with both orientations in the sum. 

We say that a function $v : V(G) \rightarrow \R$ is harmonic in $X \subseteq V(G)$ if
\begin{equation*}
    v(x) = \sum_{y \sim x} \frac{\weight(x,y)}{\weight(x)} v(y) \qquad \forall x \in X,
\end{equation*}
where $\weight(x) = \sum_{y \sim x} \weight(x,y)$.
Suppose that $v$ is harmonic in $V(G)\setminus(A\cup B)$, $v(a)=c$ for some constant $c>0$ and all $a\in A$, and $v(b)=0$ for all $b\in B$. Then the function $\theta$ defined by
\begin{equation*}
    \theta(e^-,e^+) := \weight(e^-, e^+) ( v(e^-) - v(e^+)).
\end{equation*}
is a flow from $A$ to $B$: it is antisymmetric and has zero net flow at every vertex in $V(G)\setminus(A\cup B)$ (the total flow out of $A$ need not equal $1$). In this setting, $v$ is called the \textit{voltage}, and the corresponding flow $\theta$ the associated \textit{current}.
% If $v$ is harmonic in $V(G) \setminus (A \cup B)$, $v(a) = c$, for some constant $c >0$ and all $a \in A$, and $v(b) = 0$ for all $b \in B$, then the function $\theta$ defined by
% \begin{equation*}
%     \theta(e^-,e^+) := \weight(e^-, e^+) ( v(e^-) - v(e^+)).
% \end{equation*}
% is a flow from $A$ to $B$, in the sense that it is antisymmetric and has zero net flow at every vertex in $V(G) \setminus(A\cup B)$ (the flow out of $A$ need not be equal to $1$). 

\begin{definition}[Effective resistance] \label{D:effR}
	Let $A, B \subset V$ be two disjoint sets of vertices. We define the effective resistance between $A$ and $B$ as
	\begin{equation*}
		\effR{(G,\weight)}{A}{B} := \min \big\{  \Vert\theta \Vert^2_r \, : \, \theta \text{ is a unit flow from } A \text{ to } B \big\}.
	\end{equation*}
	%and will occasionally drop the superscripts when there is no ambiguity. 
\end{definition}

\noindent It can be shown that the effective resistance is symmetric with respect to $A$ and $B$, and that there is a unique minimizing unit flow whose energy is equal to the effective resistance. We shall henceforth refer to this flow as the \textit{unit current flow}; it is induced by a corresponding voltage. %It can be shown that the effective resistance is symmetric with respect to $A$ and $B$, and that there is a unique minimizing unit flow whose energy is equal to the effective resistance, which we shall henceforth refer to as the  \textit{unit current flow} {(which has some corresponding voltage)}. 
We note that (see \cite[Section~2.2 and 2.4]{LP16})
the effective resistance may equivalently be defined as
\begin{equation} \label{eq:def_eff_R_RW}
    \effR{(G,\weight)}{A}{B} = \frac{1}{\sum_{a \in A} \weight(a) Q_a(\tau_B < \tau_A^+)}, 
\end{equation}
where $Q_a$ is the probability measure under which a random walk started at $a$ is defined, and $\tau_B$ (resp.\ $\tau^+_A$) are the first hitting (resp.\ first return) time of the random walk to $B$ (resp.\ $A$). Here the transition probabilities $q(\cdot, \cdot)$ of the random walk on $(G, \weight)$ are defined by $q(u,v) = \weight(u,v)/\weight(u)$.

One approach we will take to estimate the effective resistance is to compare it with the effective resistance on a simpler graph where the computations are significantly easier. A key tool we will repeatedly use is Rayleigh’s monotonicity principle (see \cite[Section~2.4]{LP16}).

\begin{theorem}[Rayleigh’s monotonicity principle]
    Let $\weight, \weight'$ be weights on the edges of a graph $G$ with $\weight(e) \leq \weight'(e)$ for all edges $e \in E(G)$. If $A,B \subset V$ are two disjoint sets, then
\begin{equation*}
    \effR{(G, \weight')}{A}{B} \leq \effR{(G,\weight)}{A}{B},
\end{equation*}
i.e.\ the effective resistance is a non-increasing function of the weights.
\end{theorem}

One application of Rayleigh’s monotonicity principle is to lower bound the effective resistance by that on the contracted graph $G / f$, $f = \{u, v\}$, where we identify the vertices $u$ and $v$ into a single vertex, remove $f$, and preserve all other edges. This can be thought of as giving the edge $f$ infinite weight. %or conditioning on the edge being contained in the spanning tree. %From this, one may also deduce that, given any fixed weights on the edges, the events of distinct edges being in the spanning tree are negatively correlated under the corresponding (weighted) uniform spanning tree measure. 
Two further tools that are useful for computing effective resistances are the series and parallel laws (see \cite[Section~2.3]{LP16}), which allow one to simplify paths and parallel edges into single edges with new weights without changing the effective resistance one aims to compute. %, and the Nash-Williams inequality.

One important connection between electric networks and USTs is Kirchhoff's formula,
where in the following we define for a weighted graph $(G, \weight)$ the measure
\begin{equation*}
    \bP^{(G,\weight)}(\cT = T) = \frac{1}{Z^{(G,\weight)}} \prod_{e \in T} \weight(e)
\end{equation*}
as the (weighted) uniform spanning tree measure on the set of spanning trees on $G$ similarly to \eqref{eq:PomegaT}. For notational simplicity, we will often omit the dependence on $(G, \weight)$ and write $\bP$ instead.

% Section~ 4.2 in lp - S -
\begin{theorem}[Kirchhoff's formula {\cite[Section~4.2]{LP16}}] \label{T:Kirchhoff}
	Given a weighted graph $(G, \weight)$, the probability that an edge $\{u,v\}$ is in the (weighted) uniform spanning tree satisfies
	\begin{equation} \label{eq:p_in_tree_r_eff}
        \bP^{(G,\weight)}  \big( \{u,v\} \in \mathcal{T} \big) = \weight(u,v) \effR{(G,\weight)}{u}{v}.
	\end{equation}
\end{theorem}

\noindent Applying Kirchhoff's formula to the overlap $\mathcal{O}(\beta)$ from \eqref{eq:def_overlap} shows that
\begin{align} \label{eq:OV_effR}
	\mathcal{O}(\beta)  &= \bE^{\otimes 2} \Big[ \sum_{\{u,v\} \in E} 1_{\{u,v\} \in \cT, \{u,v\} \in \cT'}  \Big] \nonumber \\
	&= \sum_{\{u,v\} \in E} \bP\big( \{u,v\} \in \cT \big)^2= \sum_{\{u,v\} \in E} \weight(u,v)^2 \effR{(G,\weight)}{u}{v}^2.
\end{align}
In fact, the theory of electric networks allows one to calculate the probability of any collection of edges being in the tree. We formalize this in the following definition and theorem, where we recall that each edge $e$ in $E$ corresponds to an edge with fixed orientation in $E^\rightarrow_{\textrm{fix}}$ with starting and end points $e^-$ and $e^+$.

\begin{definition}[Transfer-impedance matrix]
	Let $v_e(\cdot)$ be the voltages on the vertices when a unit current flow $\theta_e(\cdot)$ is sent from $e^-$ to $e^+$, i.e.\ $v_e(\cdot)$ is the unique (weighted) harmonic function on $V \setminus \{e^-, e^+\}$ with boundary values $v_e(e^-) = \effR{(G,\weight)}{e^-}{e^+}$ and $v_e(e^+) = 0$. The transfer-impedance matrix $Y = (Y(e,f))_{e,f \in E}$ is a matrix indexed by $E \times E$ with entries given by
	\begin{equation*}
		Y(e,f) := \weight(f) \big[ v_e(f^-) - v_e(f^+) \big] = \theta_e(f^-, f^+).
	\end{equation*}
\end{definition}

% Note that we have
% \begin{equation*}
	%     Y(e,f) \weight(e) = Y(f,e) \weight(f),
	% \end{equation*}
% implying that $Y(e,f)$ is symmetric in the unweighted case. 

\begin{remark} \label{R:flow}
    As $v_e(\cdot)$ is harmonic on $V \setminus \{e^-, e^+\}$, we have by the maximum principle that $v_e(e^-) \geq v(x) \geq v_e(e^+)$ for $x \in V$. Hence, if either $f^- = e^-$ or $f^+ = e^+$, then we have that %then along $f$ flow is going either out of $e^-$ or into $e^+$, 
    $Y(e,f) \geq 0$. Similarly, if either $f^- = e^+$ or $f^+ = e^-$, then $Y(e,f) \leq 0$. %(this may also be deduced by the antisymmetry property of flows). 
    In particular, if $e$ and $f$ share an endpoint, the sign of $Y(e,f)$ is determined by the fixed orientation we chose for the edges.
\end{remark}
\noindent  The following theorem first appeared in \cite{BP93}.

\begin{theorem}[Transfer-impedance theorem] \label{T:transfer-imp}
	For distinct edges $e_1, \ldots, e_k$
	\begin{equation}
		\bP(e_1, \ldots, e_k \in \cT) = \det \big( Y(e_i, e_j)_{1 \leq i,j\leq k} \big),
	\end{equation}
	where $Y(e_i, e_j)_{1 \leq i,j\leq k}$ are entries of the transfer-impedance matrix $Y$.
\end{theorem}

% Section~ 4.2 in lp - E -

\noindent Using the reciprocity law $Y(e,f)\weight(e) = Y(f,e) \weight(f)$ (see \cite[equation (2.12)]{LP16}), gives with the transfer-impedance theorem that for $e \neq f$
\begin{equation*}
    \bP(e,f \in \cT) = Y(e,e)Y(f,f) - Y(e,f)Y(f,e) = \bP(e \in \cT) \bP(f \in \cT) -  \frac{\weight(e)}{\weight(f)} Y(e,f)^2,
\end{equation*}
and hence
\begin{equation} \label{eq:Yef_as_UST}
	|Y(e,f)| = \sqrt{\frac{\weight(f)}{\weight(e)}} \sqrt{\bP(e \in \cT) \bP(f \in \cT) - \bP(e,f \in \cT)}.
\end{equation}
Furthermore, by the spatial Markov property (see \cite[Lemma 2.1]{MSS23}) and Kirchhoff's formula, we have
    \begin{equation*}
        \bP(e,f \in \cT) = \bP(e \in \cT \mid f \in \cT) \bP(f \in \cT) = \weight(e)  \effR{(K_n / f,\weight)}{e^-}{e^+}  \weight(f)  \effR{(K_n,\weight)}{f^-}{f^+},
    \end{equation*}
    where $K_n / f$ is the contracted graph with the endpoints of $f$ identified into a single vertex. The values of $|Y(e,f)|$ can therefore be expressed in terms of weights of edges and effective resistances between pairs of vertices.

%%%%%%%%%%%%%%%%%%%%%%%%%%%%%%%%%%%%%%%%%%%%%%%%%%%%%%

\subsection{Coupling to a random graph} \label{SS:coupleRG}

For $p \in [0,1]$ let $G_{n,p}$ be the Erdős-R\'enyi random graph obtained by independently keeping each edge of the complete graph with probability $p$ and deleting it otherwise. We identify $G_{n,p}$ with its set of edges, and we shall construct $G_{n,p}$ by coupling the random graph to the i.i.d.\ uniforms $(\omega_e)_{e \in E}$ by letting 
\begin{equation*}
	e \in G_{n,p} \iff \omega_e \leq p.
\end{equation*}
Further, denote $\mathcal{C}_1(p), \mathcal{C}_2(p), \ldots$ the largest connected components ordered in decreasing size. The following theorem gives an upper bound on the size of the second-largest component for a wide range of $p$, where we remark that the typical size of the 2nd largest component at $p=(1 \pm \epsilon)/n$, with $n^{-1/3} \ll \epsilon = o(1)$, corresponds to $L=2$ and $\log n$ replaced with $\log(\epsilon^3 n)$; see e.g.\ \cite{NP07}.

\begin{theorem} \label{T:C2_size}
	There exists constants $L, \eta, D, C > 0$ such that for any $n^{-1/10} \leq \epsilon \leq D$ and %!!!!!!
	\begin{equation*}
		p \not\in \Big[ \frac{1 - \epsilon}{n}, \frac{1+\epsilon}{n}\Big],
	\end{equation*}
	we have
	\begin{equation}
		\P \Big( \big| \mathcal{C}_2(p) \big| \geq L \log(n) \epsilon^{-2} \Big) \leq C n^{-\eta}.
	\end{equation}
\end{theorem}
\begin{proof}[Proof Sketch]
	The case $0 \leq pn \leq 1- \epsilon$ follows by comparison to a (slightly) subcritical branching process, together with large deviation bounds on binomial random variables (see e.g.\ Section~4.3.2 in \cite{vdH17}). In this case, the bound actually applies to the size of the largest component $\cC_1$. %and the proof thereof). 
    For the regime $pn \geq 1 + \epsilon$, one first shows that there are no components of intermediate size (see e.g.\ Proposition~4.12 of \cite{vdH17}). %where, in their notation, we may take $k_n = L \log(n) \varepsilon^{-2$ for some $L$ large enough})
    Then one shows that the giant component is unique (see e.g.\ Lemma~4.14 of \cite{vdH17} or Lemma~3 in \cite{Luc90}).
\end{proof}

When the parameter $p$ starts to become very large, then the size of the 2nd largest component becomes smaller than logarithmic order. Indeed, for $pn = \log n + c$ the probability of a random graph being connected converges to $1$ as $c \rightarrow \infty$. We will use the following theorem, and refer to Theorem~9 in Chapter 7 of \cite{Bol98} for a proof.

\begin{theorem}
	\label{T:Gnp_connect}
	Let $pn = \log n + \kappa(n)$ for $\kappa(n) \rightarrow \infty$ and $\kappa(n) \leq \sqrt{n}$. Then for $n$ large enough
	\begin{equation*}
		\P( G_{n,p} \textrm{ is not connected} \,) \leq 4 e^{-\kappa(n)} .
	\end{equation*}
\end{theorem} 

\noindent As we will see later in the proofs, once the random graph $G_{n,p}$ is connected, the effective resistance between any two vertices in the unweighted (i.e. where all weights are set equal to $1$) graph will concentrate around a deterministic value.

\subsection{Minimum spanning tree} \label{SS:MST}

Given distinct weights $\weight(\cdot)$ on the edges, the \textit{minimum spanning tree} (MST) is the unique tree that minimizes
\begin{equation*}
	\sum_{e \in T} \weight(e)
\end{equation*}
over all $T \in \T(K_n)$, the set of spanning trees in $K_n$, and we write $M = M(\weight)$ for this tree. By letting $\weight(e) = \omega_e$, we can couple the MST to the environment so that the tree with the largest probability in \eqref{eq:PomegaT} is exactly the MST. This in particular also gives a coupling of the MST to the random graph process defined in the previous section. By using Kruskal's algorithm (see e.g. the introduction of \cite{Add13}) to construct the MST, one can see that under this coupling the vertex sets of the connected components of $G_{n,p}$ and $G_{n,p} \cap M$ coincide. This allows us to prove the following lemma as a consequence of Theorem~\ref{T:Gnp_connect}. 

\begin{lemma} \label{L:max_deg_MST}
	Let $M$ be the random MST of the complete graph with $\weight(e) = \omega_e$ where $(\omega_e)_{e \in E}$ are i.i.d.\ uniforms on $[0,1]$. Denote by $\deg_M(v)$ the degree of $v$ in $M$, then
	\begin{equation*}
		\P \big( \exists v \in V : \deg_M(v) > 60 \log n \big) = O(n^{-4}).
	\end{equation*}
\end{lemma}
\begin{proof}
	Let $p = 5 \log n / n$ and $d = 60 \log n$. If $G_{n,p}$ is connected, then no other edges will be added to the MST in the coupling above, and the maximum degree of a vertex in $M$ is bounded by the maximum degree of a vertex in $G_{n,p}$. Hence, by a union bound followed by an application of Theorem~\ref{T:Gnp_connect}
	\begin{align*}
		\P \big( \exists v \in V : \deg_M(v) > d \big) &\leq \P \big( \exists v \in V : \deg_{G_{n,p}}(v) > d \big) + \P \big( G_{n,p} \textrm{ is not connected} \, \big) \\
		&\leq n \binom{n-1}{d} p^d + O(n^{-4}) \leq n \cdot \Big( \frac{e \cdot 5 \log n}{d} \Big)^d + O(n^{-4}) \\
		&\leq \exp \Big( d \big( 2 + \log 5 + \log \log n - \log d \big) \Big) + O(n^{-4}) \\
		&\leq \exp \big( d( 4  - \log 60) \big) + O(n^{-4}) = O(n^{-4}). \qedhere
	\end{align*}
\end{proof}

\noindent We remark that Lemma~\ref{L:max_deg_MST} implies that with high probability the volume of a ball of radius $r$ in the MST grows at most as fast as $ 60^r (\log n)^r$, which later allows us to apply a union bound. We expect this bound to be far from optimal; see also Remark~\ref{R:MST_growth}.

\subsection{Local Limits} \label{SS:local} 

For $v \in V(G)$ the pair $(G,v)$ is called a rooted graph with root $v$. We say that two rooted graphs $(G, v)$ and $(G', v')$ are isomorphic if there exists a graph isomorphism $\phi : V(G) \rightarrow V(G')$ between $G$ and $G'$ fixing the roots, i.e.\ there exists a bijection $\phi$ between the vertex sets such that $\phi(v) = v'$ and $\{x,y\} \in E(G)$ if and only if $\{\phi(x), \phi(y)\} \in E(G')$, and in this case we write $(G,v) \simeq (G',v')$. For $r = 1, 2, \ldots$ denote by $B_G(v,r)$ the closed ball (viewed as a rooted subgraph) of radius $r$ centered at $v$ in $G$. 

Let $(G,v)$ be a random rooted graph, and consider a sequence of random rooted graphs
$(G_n, v_n)_{n \in \N}$, where each $v_n$ is drawn uniformly from the vertex set of $G_n$.
We assume that each of these random rooted graphs is defined on some probability space,
and we write $\P$ for the underlying probability measure of any such space.
% Let $(G,v)$ be a random rooted graph, and consider a sequence of random rooted graphs
% $(G_n, v_n)_{n \in \N}$, where $(v_n)_{n \in \N}$ are vertices drawn uniformly from the vertex sets of $(G_n)_{n \in \N}$. We assume that the rooted graphs are each defined on some probability space and write $\P$ for the underlying probability measure of any of these probability spaces. 
Local (weak) convergence of $G_n$ to $G$ may be characterized by the condition that for all $r \geq 1$ and every finite rooted graph $(H,o)$
\begin{equation*}
	\limn \P\Big( B_{G_n}(v_n,r) \simeq (H,o) \Big) = \P \Big( B_G(v,r) \simeq (H,o) \Big).
\end{equation*}
We refer to \cite{BS01} for a more formal introduction.
% We refer to \cite{BS01} for a more formal introduction.
% Let $(G,v)$ be a random rooted graph defined on some probability space, and denote by $\mu$ its law on the set of finite rooted graphs. Similarly, let 
% $(G_n, v_n)_{n \in \N}$ be a sequence of random rooted graph with corresponding laws $(\mu_n)_{n \in \N}$, where $(v_n)_{n \in \N}$ are vertices drawn uniformly from the vertex sets of $(G_n)_{n \in \N}$. Local (weak) convergence of $G_n$ to $G$ may be characterized by the condition that for all $r \geq 1$ and every finite rooted graph $(H,o)$
% \begin{equation*}
% 	\limn \mu_n \Big( {\{ (G_n,v_n) :} B_{G_n}(v_n,r) \simeq (H,o) \} \Big) = \mu \Big( B_G(v,r) \simeq (H,o) \Big).
% \end{equation*}
% We refer to \cite{BS01} for a more formal introduction.

As the complete graph is transitive and the edge weights are i.i.d., the edge weights form an exchangeable sequence of random variables. In particular, the averaged law $\widehat{\P}_{n,\beta}$ is invariant under any permutation of the vertices, so we may as well fix the random vertex $v_n = 1$. In particular, the local weak convergence of the RSTRE to some rooted graph $(\mathcal{S},v)$ %with corresponding law $\mu$ 
is equivalent to
\begin{equation} \label{eq:def_local_conv}
	\widehat{\P}_{n,\beta} \Big( B_{\cT}(1,r) \simeq (H,o) \Big) \xrightarrow{n \rightarrow \infty} \P \Big( B_{\mathcal{S}}(v,r) \simeq (H,o) \Big),
\end{equation}
for all $r \geq 1$ and rooted graphs $(H,o)$. In this case we write $(\cT, 1) \xrightarrow{d} (\mathcal{S},v)$.

\smallskip

Recall from Theorem~\ref{T:local_limit} that $(\mathcal{P}, o)$ is the tree associated to the Poisson(1) branching process conditioned to survive forever with root $o$. The following result was first proven in \cite{Gri80}.
\begin{theorem}
	The UST on the complete graph converges locally to $(\mathcal{P}, o)$.
\end{theorem}
\noindent As mentioned in the introduction, since the work of \cite{Gri80}, it has been shown that $(\mathcal{P}, o)$ is also the local limit of the UST on other families of graphs, see e.g.\ \cite{BP93} and \cite{NP22}. For instance, the authors of \cite{NP22} show that the above theorem holds for any (almost) $d$-regular graph sequence with $d$ diverging as the number of vertices in the graph increases. 

\smallskip

The local limit of the MST has been studied less intensively. Nonetheless, in \cite{Add13} it was proved that the MST on the complete graph converges to a local limit (with a definition of local convergence that is slightly different from the one defined above). Furthermore, by showing that the limit has cubic volume growth compared to the quadratic volume growth of $\mathcal{P}$, the limiting object is different from $\mathcal{P}$. We note that part of their result can also be deduced from \cite[Theorem~5.4]{AS04}. The more recent work of \cite{NT24}, which also studies spectral and diffusive properties of the limiting object, shows that the convergence also takes place with respect to the local weak convergence defined at the start of the section. The following is Theorem 6.1 of \cite{NT24}, which we will apply to the complete graph $K_n$.%that in \cite{Add13} the convergence is stronger by also taking into account the edge weights of the tree.

\begin{theorem} \label{T:local_MST}
	There exists a random rooted tree $(\mathcal{M} ,o)$ such that the MST on a sequence of finite, simple, connected, regular graphs with degree tending to infinity converges locally to $(\mathcal{M}, o)$.
\end{theorem}

% \clearpage

%%%%%%%%%%%%%%%%%%%%%%% NEW (SUB-)SECTION %%%%%%%%%%%%%%%%%%%%%%%

\section{Low Disorder} \label{S:low}

In this section, we prove Theorems \ref{T:Overlap}, \ref{T:local_limit} and \ref{T:length} for $\beta = \beta(n) \ll n/\log n$. We first prove lower bounds on the effective resistance using concentration arguments in Section~ \ref{SS:effR_low}. Then we prove matching upper bounds in Section~\ref{SS:effR_up} by constructing a good flow on the weighted graph. The concentration of the effective resistance will then allow us to prove Theorem~\ref{T:Overlap} and, using the method of tree moments developed in \cite{BP93}, we prove Theorem~\ref{T:local_limit}.

%%%%%%%%%%%%%%%%%%%%%%% NEW (SUB-)SECTION %%%%%%%%%%%%%%%%%%%%%%%

\subsection{Lower Bounds on the Effective Resistance} \label{SS:effR_low}

We briefly recall a concentration inequality of \cite{MSS24} that will be useful for the lower bounds of the effective resistance between two given vertices. We consider the random weighted graph $(K_n,\weight)$ with weights given by
\begin{equation*}
	\weight(e) = \exp( - \beta \omega_e) \, \in [0,1]
\end{equation*}
where $(\omega_e)_{e \in E}$ are i.i.d.\ uniforms on $[0,1]$, whose mean is given by
\begin{equation*}
	\E[\weight(e)] = \mu = \mu(\beta) = \frac{1 - e^{-\beta}}{\beta}.
\end{equation*}
Let $e_1, \ldots, e_m$ be distinct edges in $E = E_n$, and let $S_m = \sum_{i=1}^m \weight(e_i)$ be the sum of the edge weights. Applying the Bernstein inequality (with variance of the weights bounded by $1/\beta$) as in \cite[Equation~(3.1)]{MSS24}, gives for $\beta \geq 1$ and any $ 0 < \delta \leq 1$ that
\begin{equation} \label{eq:BernU}
	\P \big( |S_m - m \mu | \geq \delta m \mu \big) \leq 2 \exp \Big( - \frac{\delta^2 m}{9 \beta } \Big).
\end{equation}
For simplicity, we shall assume from now on that $\beta \geq 1$, and remark that the concentration arguments follow similarly for $0 \leq \beta \leq 1$. %, otherwise, we have to replace \eqref{eq:BernU} with a similar concentration bound by e.g.\ applying Hoeffding's inequality (which gives good bounds up to $\beta = o(\sqrt{m})$).

For two edges $e,f$ let $|e \cap f|$ be the number of endpoints that they share. The following lemma will follow quite easily from the concentration inequality \eqref{eq:BernU} and the electric network tools developed in Section~\ref{SS:electric}. Recall that $G / f$ denotes the graph obtained by contracting the endpoints of $f$ into a single vertex.

\begin{lemma} \label{L:effR_lowdisorder_lower}
	Assume $\beta = \beta(n) \ll n/ \log n$ and let $K \geq 1$. Then there exists a constant $c > 0$ (independent of $n$ and $K$) such that the following holds with probability at least $1 - c n^{-K}$:
	\begin{enumerate}[label={\roman*)}]
		\item for all $u,v \in V$
		\begin{equation} \label{eq:effR_P_low}
			\effR{(K_n,\weight)}{u}{v} \geq \frac{2(1-o_\beta(\sqrt{K}))}{n \mu};
		\end{equation}
		\item for all edges $e = \{u,v\}$ and $f$ with $|e \cap f| = 0$
		\begin{equation} \label{eq:effR_ecapf=0_low}
			\effR{(K_n / f,\weight)}{u}{v} \geq \frac{2(1-o_\beta(\sqrt{K}))}{n \mu};
		\end{equation}
		
		\item and for all edges $e = \{u,v\}$ and $f$ with $|e \cap f| = 1$
		\begin{equation} \label{eq:effG_contract}
			\effR{(K_n / f,\weight)}{u}{v} \geq \frac{3(1-o_\beta(\sqrt{K}))}{2n \mu}.
		\end{equation}
	\end{enumerate}
\end{lemma}
\begin{proof}
	From Rayleigh's monotonicity principle, %(see \cite[Section~2.4]{LP16}), 
    it follows that the effective resistance on $K_n$ is at least the effective resistance in the multi-graph $H$ (with corresponding weight function $\weight'$) obtained by contracting the endpoints of every edge $g$ in $K_n$ with $|e \cap g| = 0$, i.e.\ edges with neither endpoint in $\{u,v\}$. Notice that in particular if $|e \cap f| = 0$, then the edge $f$ is also contracted. On the multi-graph $H$, the series and parallel law (see \cite[Section~2.3]{LP16}) give that
	\begin{equation} \label{eq:up_con}
		\frac{1}{\weight(u,v) + \Big( \big(\sum_{y \neq u,v} \weight(u,y) \big)^{-1} + \big( \sum_{y \neq u,v} \weight(y,v) \big)^{-1} \Big)^{-1}}  = \effR{(H,\weight')}{u}{v} \leq \effR{(K_n,\weight)}{u}{v}.
	\end{equation}
	Fix $M > 1$, then by \eqref{eq:BernU} with $\delta = 3 \sqrt{M \beta \log n / n}$, which is equal to $o_\beta(\sqrt{M})$ whenever $\beta \ll n/\log n$, the inequality
	\begin{equation*}
		\sum_{y \neq u,v} \weight(u,y) \leq (1 + o_\beta(\sqrt{M})) n \mu  
	\end{equation*}
	holds with probability at least $1 - O(n^{-M})$. Furthermore, as trivially $\weight(u,v) \leq 1 \ll n/ \beta$, we may easily deduce inequality \eqref{eq:effR_P_low} and \eqref{eq:effR_ecapf=0_low}.
	
	For inequality \eqref{eq:effG_contract} assume w.l.o.g.\ that $f = \{u,x\}$ with $x \neq v$. First we contract all edges $g$ that satisfy both $|e \cap g| = 0$ and $|f \cap g| = 0$, i.e.\ edges that do not have $u,v$ or $x$ as an endpoint. Then contracting $f = \{u,x\}$, gives again by the series and parallel law that
	\begin{multline*}
		\frac{1}{\weight(u,v) + \weight(x,v) + \Big( \big(\sum_{y \neq u,v,x} \weight(u,y) + \weight(x,y) \big)^{-1} + \big(\sum_{y \neq u,v} \weight(y,v)\big)^{-1} \Big)^{-1}} \\ = \effR{(H,\weight')}{u}{v} \leq \effR{(K_n / f,\weight)}{u}{v}.
	\end{multline*}
	Similar concentration arguments as for \eqref{eq:up_con}, show that \eqref{eq:effG_contract} holds with probability at least $1 - O(n^{-M})$. Choosing $M=K+4$ and applying a union bound over all $O(n^4)$ choices of edges $e,f$ ensures that the above bounds hold simultaneously on an event of probability at least $1- O(n^{-K})$, which concludes the proof.
\end{proof}

%%%%%%%%%%%%%%%%%%%%%%% NEW (SUB-)SECTION %%%%%%%%%%%%%%%%%%%%%%%
\subsection{Upper Bounds on the Effective Resistance} \label{SS:effR_up}

To prove the analogous upper bounds of Lemma~\ref{L:effR_lowdisorder_lower}, we will first show upper bounds on the effective resistance on $(G_{n,p}, 1)$, the Erdős-R\'enyi random graph with weights all equal to one. Then, roughly speaking, we approximate the weighted graph by a union of disjoint Erdős-R\'enyi random graphs with discretized weights to give bounds on the effective resistance on $(K_n,\weight)$. %We shall make use via 
See also the coupling between $G_{n,p}$ and the random variables $(\omega_e)_{e \in E}$ from  Section~\ref{SS:coupleRG}.

% \begin{lemma} 
	%     If $\log n \ll pn \ll \sqrt{n}$, then for any $K > 0$ we have with probability $1 - O(n^{-K})$ that 
	%     \begin{equation} \label{eq:effGnp}
		%         \effR{G_{n,p}}{u}{v} \leq \frac{2(1+o(1))}{pn}.
		%     \end{equation}
	%     for all of vertices $(u,v)$. 
	% \end{lemma}

%(G,\weight)
%(K_n,\weight)
%(G_{n,p, 1)}

\begin{lemma} \label{L:effGnp}
	Let $K \geq 1$. There exists some universal constant $C > 0$ and a constant $c_K > 0$ depending on $K$, such that for $p = p(n)$ satisfying $\log n \ll pn \ll \sqrt{n}$ and $\xi = \xi(n) = C \sqrt{K \log n/pn}$, we have with probability at least $1 - c_K n^{-K}$:
	\begin{enumerate}[label={\roman*)}]
		\item for all $u,v \in V$
		\begin{equation} \label{eq:Gnp_effR_upper}
			\effR{(G_{n,p}, 1)}{u}{v} \leq \frac{2(1 +\xi)}{pn};
		\end{equation}
		\item and for all edges $e = \{u,v\}$, $f =\{u,x\}$ with $|e \cap f| = 1$
		\begin{equation} \label{eq:effGnp_contract}
			\effR{(G_{n,p} / f, 1)}{u}{v} \leq \frac{3(1+\xi)}{2pn}.
		\end{equation}
	\end{enumerate}
\end{lemma}

We remark that inequality \eqref{eq:Gnp_effR_upper} is essentially Proposition~3.1 in \cite{Jon98} with the additional constraint that $\xi = C \sqrt{K \log n/pn}$ for some large constant $C$, ensuring that with high probability the degrees in the random graph are of the correct size, and hence we will not prove it. The proof in \cite{Jon98} uses a branching type process to show that there are many different paths between the neighbors of $u$ and the neighbors $v$, which give rise to a flow that captures the behavior of the effective resistance. It is possible to adapt the proof to also show \eqref{eq:effGnp_contract} by noting that the degree of the first vertex is essentially twice as large as in the non-contracted case. However, instead of repeating many parts of the proof, we will make use of the following lemma.

\begin{lemma}[Lemma~3.5 in \cite{NP22}] \label{L:effR3vert}
	Let $(G, \weight)$ be a weighted graph and let $v_1, \ldots, v_k$ be distinct vertices with $k \geq 3$. Let $x > 0$ be given and $\delta > 0$ satisfy $\delta \leq \frac{x}{32k}$. Assume that for any distinct $v_i, v_j$ one has
	\begin{equation*}
		\big| \effR{(G,\weight)}{v_i}{v_j} - x \big| \leq \delta.
	\end{equation*}
	Then
	\begin{equation*}
		\Big| \effR{(G,\weight)}{v_1}{\{v_2, \ldots, v_k \}} - \frac{kx}{2(k-1)} \Big| \leq 72 k \delta.
	\end{equation*}
\end{lemma}

\begin{proof}[Proof of \eqref{eq:effGnp_contract} in Lemma~\ref{L:effGnp}]
	Restrict to the event that the bound in \eqref{eq:Gnp_effR_upper} holds for all pairs of vertices, and that the degree of each vertex is between $(1+ \xi)pn$ and $(1- \xi) pn$. The former event holds with probability at least $1 - c_k n^{-K}$ by the first part of the Lemma, and by a Chernoff-type bound, the latter event holds with a probability at least
    \begin{equation*}
        1 - 2 \exp\Big( - \frac{\xi^2 p(n-1)}{3}\Big) = 1 - 2 \exp\Big( - C^2 K \log n \frac{(n-1)}{3 n}\Big) = 1 - O(n^{-K}),
    \end{equation*}
    for a large enough $C$. Now using the known lower bound on the effective resistance
	\begin{equation*}
		\effR{(G_{n,p}, 1)}{u}{v} \geq \frac{1}{\deg_{G_{n,p}}(u) + 1} + \frac{1}{\deg_{G_{n,p}}(v) + 1},
	\end{equation*}
    which may be obtained by the Nash-Williams inequality (Lemma~\ref{L:nash_williams}) by possibly changing an edge between $u$ and $v$ into a path of length $2$ with weights equal to $2$, gives for all pairs of (distinct) vertices $u$ and $v$ that
	\begin{equation} \label{eq:effR_concentrate}
		\big| \effR{(G_{n,p}, 1)}{u}{v} - \frac{2}{pn} \big| = O\Big( \frac{ \xi}{pn} \Big)
	\end{equation}
	with probability at least $1 - O(n^{-K})$ for any $K \geq 1$ (here the implicit constant in $O(n^{-K})$ depends on $K$).
	
    Now fix distinct vertices $u,v$ and $x$, and recall that $e = \{u, v\}$ and $f =\{u,x\}$. Denote by $\Tilde{u}$ the contracted vertex corresponding to $u$ and $x$ in $K_n / f$. Notice that as $e$ and $f$ share an endpoint, we may couple a random walk $(Y_i)_{i \geq 0}$ started at $v$ in $(G_{n,p}, 1)$ with a random walk $(\Tilde{Y}_i)_{i \geq 0}$ started at $v$ in $(G_{n,p} /f , 1)$, such that until the first time either $\{u,x\}$ or $\Tilde{u}$ is hit, we have $Y_i = \Tilde{Y}_i$. By the random walk representation of the effective resistance in \eqref{eq:def_eff_R_RW}, we thus have
	\begin{equation*}
		\effR{(G_{n,p} /f , 1)}{\Tilde{u}}{v} = \effR{(G_{n,p}, 1)}{v}{\{u,x\}}.%\frac{1}{\weight(\Tilde{u}) Q_{\Tilde{u}}( \tau_{v} < \tau^+_{\Tilde{u}})} = \frac{1}{\weight(u) Q_u(\tau_{v} < \tau^+_{u}) + \weight(x) Q_x( \tau_{v} < \tau^+_{u})} \\
		%&= \effR{G}{\{u,x\}}{v} = \effR{G}{v}{\{u,x\}}.
	\end{equation*}
	Applying Lemma~\ref{L:effR3vert} for the vertices $u,v,x$ with \eqref{eq:effR_concentrate} then gives that for some fixed $C' > 0$
	\begin{equation*}
		\Big| \effR{(G_{n,p}, 1)}{v}{\{u,x\}} - \frac{3}{2 pn} \Big| \leq C'  \frac{ \xi}{pn},
	\end{equation*}
	completing the proof by possibly enlarging the constant $C$.
\end{proof}

The bounds on the effective resistance for $G_{n,p}$ almost give the correct bounds on the effective resistance of $(K_n,\weight)$. Namely, we let $ p = 1/\beta \gg \log n/ n$ and set $\weight'(e) = \exp(-\beta p) \leq \weight(e)$ if $e$ is in $E(G_{n,p})$, which by our coupling is equivalent to $\omega_e \leq p$, and $\weight'(e) = 0$ otherwise. From Definition~\ref{D:effR}, it is clear that multiplying all the weights of the edges by a constant $a > 0$ changes the effective resistances by a multiplicative factor of $1/a$. Thus, Rayleigh's monotonicity principle gives
\begin{equation*}
	\effR{(K_n,\weight)}{u}{v} \leq \effR{(G_{n,p},\weight')}{u}{v} \leq  \frac{(2 + o(1))}{np} e^{\beta p} \approx \frac{2 \cdot e}{n \mu}.
\end{equation*}
We shall improve the constant from $2e$ to $2$ using the following straightforward lemma.
\begin{lemma} \label{L:disjoint_flow}
	Let $\theta_1, \ldots, \theta_m$ be unit flows (cf.\ Section~\ref{SS:electric}) from $u$ to $v$ that are supported on disjoint edges, that is for each edge $e$ we have $\theta_i(e) > 0$ for at most one $i \in \{1, \ldots, m\}$. Then, for any $\alpha_1, \ldots, \alpha_m \geq 0$ we have
	\begin{equation*}
		\Big\Vert \sum_{i=1}^m \alpha_i \theta_i \Big\Vert^2_r = \sum_{i=1}^m \alpha^2_i \Vert\theta_i\Vert^2_r.
	\end{equation*}
\end{lemma}
\begin{proof}
	As the support of the flows are disjoint, we have that
	\begin{equation*}
		\Big\Vert \sum_{i=1}^m \alpha_i \theta_i \Big\Vert^2_r = \frac{1}{2}\sum_{e \in E} \Big( \sum_{i=1}^m \alpha_i \theta_i(e) \Big)^2 = \frac{1}{2}  \sum_{e \in E} \sum_{i=1}^m \alpha_i^2 \theta_i(e)^2 = \sum_{i=1}^m \alpha^2_i  \Vert\theta_i \Vert^2_r. \qedhere
	\end{equation*}
\end{proof} 

We are now ready to prove the following lemma.

\begin{lemma} \label{L:effR_lowdisorder_upper}
	Assume $\beta = \beta(n) \ll n/ \log n$ and let $K \geq 1$. Then there exists a constant $c_K > 0$ such that following holds with probability at least $1 - c_K n^{-K}$:
	\begin{enumerate}[label={\roman*})]
		\item for all $u,v \in V$
		\begin{equation}
			\effR{(K_n,\weight)}{u}{v} \leq \frac{2(1 + o_\beta(\sqrt{K}))}{n \mu}; \label{eq:effR_lowdisorder}
		\end{equation}
		
		\item and for all edges $e = \{u,v\}$ and $f =\{u,x\}$ with $|e \cap f| = 1$
		\begin{equation} \label{eq:effR_lowdisorder_contract}
			\effR{(K_n / f,\weight)}{u}{v} \leq \frac{3(1 + o_\beta(\sqrt{K}))}{2n \mu}.
		\end{equation}
	\end{enumerate}
\end{lemma}

\begin{proof}
	We construct a random graph $G_n[p_0, p_1)$ by letting $G_n[p_0, p_1)$ be the subgraph of $K_n$ consisting of edges satisfying
	\begin{equation*}
		e \in G_n[p_0, p_1) \iff \omega_e \in [p_0, p_1)
	\end{equation*}
	for $0 \leq p_0 \leq p_1 \leq 1$. Define $p = \kappa(n) \log n /n$ for some function $\kappa = \kappa(n) \rightarrow \infty$ such that $\beta = o(p^{-1})$. Let $m := \lfloor p^{-1} \rfloor$, and consider the sequence of snapshots 
	\begin{equation*}
		p_j = j \cdot p \qquad j=1, \ldots, m.
	\end{equation*}
	We will construct a flow on $K_n$ by considering a weighted sum of disjoint flows on $G_n[p_{j-1}, p_{j})$. Namely, for each $j$ let $\theta_j$ be the current unit flow (the unique flow minimizing the energy in Definition~\ref{D:effR}) on $G_n[p_{j-1}, p_{j})$, and let
	\begin{align*}
		\Tilde{\alpha}_j = e^{-\beta \cdot p_{j}} = e^{-j \beta p}, \qquad j&=1,\ldots, m, \\
		\alpha_j = \frac{\Tilde{\alpha}_j }{\sum_{i=1}^{m} \Tilde{\alpha}_i } = \Tilde{\alpha}_j \frac{e^{\beta p } - 1}{1 - e^{-\beta m p}}, \qquad j&=1, \ldots, m.
	\end{align*}
	As $\sum_{i=j}^m \alpha_j = 1$, the flow defined by
	\begin{equation*}
		\theta := \sum_{j=1}^m \alpha_j \theta_j
	\end{equation*}
	is a unit flow from $u$ to $v$ where each $\theta_j$ is supported on different edges. Lemma~\ref{L:disjoint_flow} then gives that
	\begin{equation} \label{eq:effR_sum_flows}
		\effR{(K_n,\weight)}{u}{v} \leq  \Vert \theta \Vert^2_r = \sum_{j=1}^m \alpha^2_j  \Vert\theta_j \Vert^2_r.
	\end{equation}
	
	By Rayleigh's monotonicity principle, the effective resistance on $G_n[p_{j-1}, p_{j})$ is upper bounded by the effective resistance on the same graph where each weight is decreased to $\exp( - j \beta p)$. Thus by inequality \eqref{eq:Gnp_effR_upper} of Lemma~\ref{L:effGnp} (with $\xi = C \sqrt{K+1} \kappa^{-1/2} = o_\beta(\sqrt{K})$) we have
	\begin{equation*}
		\Vert\theta_j \Vert^2_r \leq \frac{2(1 +o_\beta(\sqrt{K}))}{pn} e^{j \beta p}.
	\end{equation*}
	for all $j$ with probability at least $1 - O(n^{-K})$. Inserting this bound into \eqref{eq:effR_sum_flows} for each $j$ gives that
	\begin{align}
		\effR{(K_n,\weight)}{u}{v} &\leq \frac{2(1 +o_\beta(\sqrt{K}))}{pn} \sum_{j=1}^m \frac{\Tilde{\alpha}_j^2 }{\big( \sum_{j=1}^m \Tilde{\alpha}_j \big)^2 } e^{j \beta p} \nonumber \\
		&= \frac{2(1 +o_\beta(\sqrt{K}))}{pn} \Big( \frac{e^{\beta p } - 1}{1 - e^{-\beta m p}} \Big) \nonumber \\
        &\leq \frac{2(1 +o_\beta(\sqrt{K}))}{pn} \frac{1}{1 - p} \Big( \frac{e^{\beta p } - 1}{1 - e^{-\beta}} \Big), \label{eq:before_Taylor}
	\end{align}
    where the last line uses $mp \geq 1 - p$ and that $x \mapsto 1 - e^{-x}$ is concave, so that
    \begin{equation*}
        1 - e^{-\beta mp} \geq 1 - e^{-\beta(1-p)} \geq (1-p)(1-e^{-\beta})
    \end{equation*}
    whenever $p \leq 1$. As $\beta p = o_\beta(1)$, a Taylor expansion gives $e^{\beta p} - 1 = \beta p + O(\beta^2 p^2)$, which applied to \eqref{eq:before_Taylor} yields, after some algebraic simplifications, the inequality \eqref{eq:effR_lowdisorder}. The proof of \eqref{eq:effR_lowdisorder_contract} follows along the same lines by using inequality \eqref{eq:effGnp_contract} in Lemma~\ref{L:effGnp} instead of the inequality in \eqref{eq:Gnp_effR_upper}.
\end{proof}

We may now easily prove Theorem~\ref{T:Overlap} and Theorem~\ref{T:length} for the low disorder regime of $\beta \ll n/\log n$. 

\begin{proof}[Proof of Theorem~\ref{T:Overlap}~\ref{T:Overlap_low}]
	By equation \eqref{eq:OV_effR}, Lemma~\ref{L:effR_lowdisorder_lower} and Lemma~\ref{L:effR_lowdisorder_upper} with some large but fixed $K$ (so that $o_\beta(\sqrt{K}) = o_\beta(1)$), we have that for $\beta \ll n/\log n$
	\begin{equation*}
		\mathcal{O}(\beta) = (1 + o_\beta(1)) \frac{4}{n^2 \mu^2} \sum_{e \in E} \weight(e)^2 
	\end{equation*}
	with high probability. Using the concentration bound of \eqref{eq:BernU} for the sum with $\beta' = 2 \beta$, $m = |E_n| = n(n-1)/2$ and, say, $\delta = n^{-1/10}$ gives
    \begin{equation*}
        \P\Big( \big| \sum_{e \in E} \weight(e)^2 - \frac{m(1-e^{\beta'})}{\beta'} \big| \geq \delta m \frac{1 - e^{\beta'}}{\beta'} \Big) \leq 2 \exp \big( - \frac{\delta^2 n^2}{9 \beta'} \big) \leq 2 \exp(-\sqrt{n}),
    \end{equation*}
    where the last inequality holds when $n$ is large enough. Therefore, using that $\delta = o(1)$ gives with high probability 
    \begin{equation*}
        \mathcal{O}(\beta) = (1 + o_\beta(1)) \frac{4}{n^2 \mu^2} \frac{m(1-e^{2\beta})}{2\beta} =  (1 + o_\beta(1)) \beta \frac{(1-e^{-2 \beta})}{(1-e^{-\beta})^2},
    \end{equation*}
    which finishes the proof. 
\end{proof}

\begin{proof}[Proof of Theorem~\ref{T:length}~\ref{T:length_low}]
    Using Kirchhoff's formula (Theorem \ref{T:Kirchhoff}), the expected (w.r.t.\ $\bE$) edge length satisfies
	\begin{equation*}
		\bE \big[ L(\cT) \big] = \sum_{\{u,v\} \in E} \omega_{\{u,v\}} \bP\big( \{u,v\} \in \cT \big) =  \sum_{\{u,v\} \in E} \omega_{\{u,v\}} e^{-\beta \omega_{\{u,v\}}} \effR{(K_n,\weight)}{u}{v}.
	\end{equation*}
	Therefore, by Lemma~\ref{L:effR_lowdisorder_lower} and Lemma~\ref{L:effR_lowdisorder_upper} with some large but fixed $K$, we have 
	\begin{equation}
		\bE \big[ L(\cT) \big] = (1+o_\beta(1)) \frac{2}{n \mu} \sum_{e \in E} \omega_{e} e^{-\beta \omega_e} \label{eq:low_length_concentrate}
	\end{equation}
	with probability, say, at least $1 - O(n^{-2})$. Let $X_{e} = \omega_e e^{-\beta \omega_e}$, then 
	\begin{equation*}
		\E[X_e] = \frac{1 - \beta e^{-\beta} - e^{-\beta}}{\beta^2},
	\end{equation*}
	and there exists a constant $C > 0$ such for $\beta \geq 1$, we have 
	\begin{align*}
		{\rm  Var} (X_e) \leq C \frac{1}{\beta^3}.
	\end{align*}
	As \eqref{eq:low_length_concentrate} holds with probability at least $1 - O(n^{-2})$ and the length of $\cT$ is trivially between $0$ and $n-1$, the equality in \eqref{eq:length_low} holds in $\P$-expectation. Furthermore, using Bernstein's inequality similarly to \eqref{eq:BernU}, with $m = n(n-1)/2$ and $m/\beta^2 \rightarrow \infty$, one may show that the right-hand side of \eqref{eq:low_length_concentrate} is, with high probability, as in \eqref{eq:length_low}. 
\end{proof}

%%%%%%%%%%%%%%%%%%%%%%% NEW (SUB-)SECTION %%%%%%%%%%%%%%%%%%%%%%%

\subsection{Transfer-impedance}

Recall the definition of the transfer-impedance matrix from Section~\ref{SS:electric} and that $|e \cap f|$ was defined to be the number of endpoints that edges $e$ and $f$ share. We denote by $P^0(\cdot)$ and $Y^0$ the unweighted UST measure and transfer-impedance matrix on the complete graph $(K_n, 1)$, respectively. We shall make use of the following lemma.

\begin{lemma} \label{L:Y^0P^0}
    On the unweighted complete graph $(K_n, 1)$ we have that
    \begin{enumerate}[label={\roman*})]
        \item \label{enu:Y^0} the transfer-impedance matrix satisfies
        \begin{equation*}
	| Y^0(e,f)| =
	\begin{cases}
		2/n & \text{if } |e \cap f| = 2, \\
		1/n & \text{if } |e \cap f| = 1, \\
		0 & \text{if } |e \cap f| = 0,
	\end{cases}
        \end{equation*}

        \item \label{enu:P^0} and for distinct edges $\{e_1, \ldots, e_k\}$ not containing a cycle
        \begin{equation} 
		\bP^0(e_1, \ldots, e_k \in \cT) \geq \frac{1}{k! n^k}. \label{eq:k!_UST_bound}
	\end{equation}
        
    \end{enumerate}
\end{lemma}%
\begin{proof}[Proof sketch]
    Note that by symmetry, for any edge $e \in E$ one has $P^0(e \in \cT) = 2/n$ so that by Kirchhoff's formula the effective resistance between two endpoints of an edge is $2/n$. Item \ref{enu:Y^0} then follows by observing that the unique harmonic function $v_e(\cdot)$ with $v_e(e^-) = 2/n$ and $v_e(e^+) = 0$ is given by $v_e(x) = 1/n$ for $x \not\in \{e^-, e^+\}$.

    For \ref{enu:P^0}, we first note that on the contracted graph $G / \{e_1, \ldots, e_j\}$ we have by the Aldous-Broder algorithm (see e.g.\ \cite{Ald90} or \cite{Bro89}) that
    \begin{equation*}
        \bP^0( e_{j} \in \cT \mid e_1, \ldots, e_{j-1} \in \cT) \geq \frac{1}{j n}. 
    \end{equation*}
    Indeed, on $G / \{e_1, \ldots, e_{j-1}\}$ any vertex has at most $jn$ many neighbors, so that with probability at least $1/jn$ the random walk started at an endpoint of $e_{j}$ immediately uses $e_{j}$, which by the algorithm implies that $e_{j}$ is in the tree. By splitting into conditional probabilities, we get that
    \begin{equation*} 
		\bP^0(e_1, \ldots, e_k \in \cT) = \prod_{j=1}^k \bP^0( e_{j} \in \cT \mid e_1, \ldots, e_{j-1} \in \cT) \geq \frac{1}{k! n^k}. \qedhere
	\end{equation*}
\end{proof}

In the following lemma, we give sharp bounds on the absolute values of the entries of the transfer-impedance matrix in the weighted case.

\begin{lemma} \label{L:Y(e,f)}
	If $\beta = \beta(n) \ll n/\log n$, then for any $K \geq 1$ and all edges $e,f$
	\begin{equation} \label{eq:Y(e,f)_abs}
		|Y(e,f)| = \frac{\weight(f)}{n \mu} \big( |e \cap f| + o_\beta(\sqrt{K}) \big).
	\end{equation}
	with probability at least $1 - c_K n^{-K}$ for some constant $c_K > 0$.  In particular,
	\begin{equation} \label{eq:Y(e,f)-Y^0(e,f)}
		\Big| Y(e,f) - \frac{\weight(f)}{\mu} Y^0(e,f) \Big| = o_\beta(\sqrt{K}) \cdot \frac{\weight(f)}{n\mu}.
	\end{equation}
	with probability at least $1 - c_K n^{-K}$.
\end{lemma}
\begin{proof}
	Recall from Section~\ref{SS:electric} that
	\begin{equation*}
		Y(e,e) = \bP(e \in \cT) = \weight(e) \effR{(K_n,\weight)}{e^-}{e^+},
	\end{equation*}
	and, by equation \eqref{eq:Yef_as_UST}, for $e,f$ with $|e \cap f| \leq 1$
	\begin{align}
		Y(e,f)^2 &= \frac{\weight(f)}{\weight(e)} \big( \bP(e \in \cT) \bP(f \in \cT) - \bP(e,f \in \cT) \big) \nonumber \\
		&= \weight(f)^2 \effR{(K_n,\weight)}{f^-}{f^+} \Big( \effR{(K_n,\weight)}{e^-}{e^+}  - \effR{(K_n / f,\weight)}{e^-}{e^+} \Big). \label{eq:Y(e,f)^2}
	\end{align}

	Rayleigh's monotonicity principle shows that for edges with $|e \cap f| = 0$ we have
	\begin{equation*}
		\effR{(K_n / f,\weight)}{e^-}{e^+} \leq \effR{(K_n,\weight)}{e^-}{e^+},
	\end{equation*}
	so that the bounds from Lemma~\ref{L:effR_lowdisorder_lower} and Lemma~\ref{L:effR_lowdisorder_upper} give for $|e \cap f| \leq 1$
	\begin{align*}
		\effR{(K_n,\weight)}{e^-}{e^+} &= (1 + o_\beta(\sqrt{K})) \frac{2}{ n \mu}, \\
		\effR{(K_n / f,\weight)}{e^-}{e^+} &= (1 + o_\beta(\sqrt{K})) \Big( \frac{4 - |e \cap f|}{2 n \mu} \Big),
	\end{align*}
	with probability at least $1 - O(n^{-K})$. Inserting this back into \eqref{eq:Y(e,f)^2} readily gives \eqref{eq:Y(e,f)_abs}.
	
	For equation \eqref{eq:Y(e,f)-Y^0(e,f)}, notice that by the observation in Remark~\ref{R:flow}, the signs of $Y(e,f)$ and $Y^0(e,f)$ agree as long as $|e \cap f| \geq 1$. %Addtionally, we have $Y^0(e,f) = 0$ whenever $|e \cap f| = 0$. %Indeed, any edge pointing away from $e^-$ or towards $e^+$ can only have non-negative flow, while flow is non-positive for edges pointing towards $e^-$ or away from $e^+$. 
    Hence, in view of Lemma~\ref{L:Y^0P^0}~\ref{enu:Y^0}, the equality of \eqref{eq:Y(e,f)_abs} implies the equality of \eqref{eq:Y(e,f)-Y^0(e,f)}. %regardless of $|e \cap f|$. 
\end{proof}

%Denote by $\bP^0$ the unweighted UST measure on $K_n$.
Lemma~\ref{L:Y(e,f)} has the following corollary.  
\begin{corollary} \label{C:ek_in_T}
	Fix $k \geq 1$ and let $e_1, \ldots, e_k \in E$ be distinct edges. If $\beta = \beta(n) \ll n/\log n$, then for any $K \geq 1$
	\begin{equation*}
		\bP(e_1, \ldots, e_k \in \cT) = \Big(1 +  k (k!)^2 o_\beta \big( \sqrt{K} \big) \Big) \prod_{i=1}^k \frac{\weight(e_i)}{\mu} \bP^0(e_1, \ldots, e_k \in \cT) \label{eq:ek_in_T}
	\end{equation*}
	with probability at least $1 - c_K n^{-K}$ for some constant $c_K > 0$.
\end{corollary}
\begin{proof}
	If $e_1, \ldots e_k$ contains a cycle, then both sides of the equation are $0$ and the equality holds deterministically. Hence, assume that $e_1, \ldots, e_k$ contains no cycles. By the transfer-impedance theorem (Theorem~\ref{T:transfer-imp}) and Lemma~\ref{L:Y(e,f)}, we have with probability $1 - O(n^{-K})$ that
	\begin{align*}
		\bP(e_1, \ldots, e_k \in \cT) &= \sum_{\sigma \in S_k} \Big( sgn(\sigma) \prod_{i=1}^k Y(e_i, e_{\sigma(i)})\Big) \\
		&= \sum_{\sigma \in S_k} \bigg[ sgn(\sigma)  \prod_{i=1}^k \frac{\weight(e_{i})}{\mu}\prod_{i=1}^k \Big(Y^0(e_i, e_{\sigma(i)}) + \frac{o_\beta(\sqrt{K})}{n} \Big) \bigg] \\
		&= \prod_{i=1}^k \frac{\weight(e_{i})}{\mu} \bigg[ \Big( \sum_{\sigma \in S_k}  sgn(\sigma) \prod_{i=1}^k Y^0(e_i, e_{\sigma(i)})  \Big) + k! \frac{k \cdot o_\beta(\sqrt{K})}{n^k}  \bigg] \\
		&= (1 + k (k!)^2 o_\beta(\sqrt{K})) \prod_{i=1}^k \frac{\weight(e_i)}{\mu} \bP^0(e_1, \ldots, e_k \in \cT),
	\end{align*}
	where using \eqref{eq:k!_UST_bound} from Lemma~\ref{L:Y^0P^0}~\ref{enu:P^0} we can absorb the $k! k n^{-k} o_\beta(\sqrt{K})$ error term into $k (k!)^2 o_\beta(\sqrt{K})$ times the probability of the edges being in the unweighted UST.
\end{proof}

\noindent We remark that in view of \eqref{eq:k!_UST_bound}, to make Corollary~\ref{C:ek_in_T} useful we must pick $K > k$.

%%%%%%%%%%%%%%%%%%%%%%% NEW (SUB-)SECTION %%%%%%%%%%%%%%%%%%%%%%%

\subsection{Proof of Theorem~\ref{T:local_limit} for small \texorpdfstring{$\beta$}{beta} via tree moments}

For random variables $X_n$ taking values in $\R$, the method of moments says that if all moments of $X_n$ converge to the moments of $X$ and these moments do not grow too fast, then $X_n$ converges in distribution to $X$. In \cite{BP93}, the authors have generalized this result to $t$\textsuperscript{th} tree moments and local convergence of trees.
%\footnote{ We remark that the convergence in \cite{BP93} considers fixed roots instead of uniformly chosen random roots as Section~\ref{SS:local}, however, as the complete graph is transitive, these notions are equivalent.}. 
More precisely, for a finite rooted tree $(t,\rho)$ and a rooted graph $(G,o)$, a map $f$ from $V(t)$ to $V(G)$ is called a \textit{tree-map} if $f$ is injective, maps $\rho$ to $o$, and if $u \sim v$ in $t$ then $f(u) \sim f(v)$ in $G$. Denote by $N(G,o;t,\rho)$ the number of distinct tree-maps between $(t,\rho)$ and $(G,o)$. %We note that $N(G,t)$ also depends on the choice of roots $s_1, s_2$, but we will omit this dependence in the notation. 
Suppose $(\cT_n, o_n)$ for $n \geq 1$ and $(\cT_\infty, o_\infty)$ are random rooted trees, each defined on some
probability space. In what follows, we use the symbol $\P$ for the underlying probability measure of any of those probability spaces.
% Let $(\bP_n)_{n \geq 1}$ and $\bP_\infty$ be probability distributions on the set of rooted trees, with corresponding expectations $(\bE_n)_{n \geq 1}$ and $\bE_\infty$. We write $((\cT_n, o_n))_{n \geq 1}$ and $(\cT_\infty, o_\infty)$ for rooted trees with respective laws $(\bP_n)_{n \geq 1}$ and $\bP_\infty$. 
The following is an analog of the method of moments for finite trees.

\begin{proposition}[Proposition~8.7 in \cite{BP93}] \label{P:tree_moment}
	Suppose that $\E [N(\cT_n,o_n; t,\rho)] \rightarrow \E [N(\cT_\infty, o_\infty; t,\rho)] < \infty$ for all finite rooted trees $(t,\rho)$ and that the law of $(\cT_\infty, o_\infty)$ is uniquely determined by the values of $\E [N(\cT_\infty, o_\infty; t,\rho)]$. Then $(\cT_n, o_n)$ converges locally to $(\cT_\infty, o_\infty)$.
\end{proposition}

One such tree that is uniquely determined by its $t$\textsuperscript{th} moments is $(\mathcal{P},o)$ the Poisson(1) tree conditioned to survive forever (see Section~8 of \cite{BP93}). Consequently, to prove the local convergence of $\cT$ in the case $\beta \ll n/\log n$, it suffices to prove convergence of the expectations of $N(\cT, 1; t,\rho)$.

\begin{proof}[Proof of Theorem~\ref{T:local_limit}~\ref{T:local_limit_low}]
	Similarly to the proof of Theorem~5.3 in \cite{BP93}, for any finite rooted tree $(t,\rho)$ and tree map $f$, we denote by $\textrm{Im}_E(f) \subseteq E(K_n)$ the edges in the image of $f$.
    That is, an edge $e = \{u, v\} \in E(K_n)$ is in $\textrm{Im}_E(f)$ if there exists adjacent vertices $x,y$ in $t$ such that $f(x) = u$ and $f(y) = v$. Let $K > |t|-1$ be fixed (so that $|t|(|t|!)^2 \cdot o_\beta(\sqrt{K}) = o_\beta(1)$), and write $\textrm{Im}_E(f) = \{e_1, \ldots, e_{|t| -1} \}$ for some edges $e_1, \ldots, e_{|t| -1} \in E(K_n)$. Corollary~\ref{C:ek_in_T} gives that
	\begin{align*}
		\E\big[ \bP( \textrm{Im}_E(f) \subseteq \cT) \big] &= \big(1 + o_\beta(1) \big) \E\Big[ \prod_{i=1}^{|t|-1} \frac{\weight(e_i)}{\mu} \Big] \bP^0 \big( \textrm{Im}_E(f) \subseteq \cT \big) + O(n^{-K}) \\
        &= \big(1 + o_\beta(1) \big) \bP^0 \big( \textrm{Im}_E(f) \subseteq \cT \big),
	\end{align*}
    where, as in \eqref{eq:k!_UST_bound}, we have used that the probability that any $|t|-1$ edges are in the (unweighted) UST is bounded below by $1/(|t|-1)! (|t|-1)^k$ so that we can absorb the $O(n^{-K})$ factor.
	Hence, summing over all tree maps $f$
	\begin{equation*}
		\E \Big[ \bE[ N(\cT,1; t,\rho)] \Big] = \E \Big[ \sum_f \bP \big( \textrm{Im}_E(f) \subseteq \cT \big)  \Big] = \big(1+ o_\beta(1) \big) \sum_f \bP^0 \big( \textrm{Im}_E(f) \subseteq \cT \big),
	\end{equation*}
	which converges to $\E[ N(\mathcal{P},o; t,\rho)]$ by Theorem~5.3 of \cite{BP93}. Applying Proposition~\ref{P:tree_moment} then ensures that we have the local convergence of $\cT$ to the Poisson(1) tree conditioned on survival, which completes the proof.
\end{proof}

% \clearpage

\section{High disorder} \label{S:high}

In this section, we will first prove Theorem~\ref{T:Overlap} and Theorem~\ref{T:length} for $\beta \gg n (\log n)^2$. Then we will adapt these techniques to show the convergence of the local limit. The key argument is that, with high probability, edges in the MST with weights outside a window of width $2 \epsilon/n$ (for some $\epsilon = o(1)$ depending on $\beta$) must also appear in the RSTRE. We remark that here high disorder refers to $\beta \gg n$ and not to the high disorder regime $\beta \gg n^{4/3}$ as in \cite{MSS24}.

\subsection{Overlap} \label{SS:overlap_high}

The edge overlap result of Theorem~\ref{T:Overlap}~\ref{T:Overlap_high}, corresponding to large $\beta$, is restated in the next lemma.
\begin{lemma} \label{L:high_overlap}
	Let $\beta = \beta(n) \gg n (\log n)^2$, then with high probability
	\begin{equation} \label{eq:high_overlap}
		\mathcal{O}(\beta) = (1 - o_\beta(1)) n.
	\end{equation}
\end{lemma}

\noindent 
The proof idea of Lemma~\ref{L:high_overlap} is the following. Suppose that the event $\{e \in$ MST$\}$ occurs for some edge $e \in E$. We first construct several different events (see \eqref{eq:defG_e}, \eqref{eq:defM_E}, \eqref{eq:defC_e} and \eqref{eq:def_B_e}), and show that each of them occurs with a high enough $\P$-probability. Then in \eqref{eq:e_in_cT_MST} we show that, on these events, the $\bP$-probability of $e$ also being in the RSTRE is close to $1$. This will imply that many edges in the MST will also be in the RSTRE with high probability, and hence the overlap must be large. Before we proceed with proof, we recall one more tool to bound effective resistances. The following inequality may be found in \cite[Section~2.5]{LP16}.

\begin{lemma}[Nash-Williams inequality] \label{L:nash_williams}
    Let $\Pi_1, \ldots \Pi_k \subset E$ be disjoint cutsets of distinct vertices $u$ and $v$, i.e.\ any path between $u$ and $v$ uses at least one edge in $\Pi_i$ for every $i \in \{1, \ldots, k\}$. Then
    \begin{equation}
        \effR{(G,\weight)}{u}{v} \geq \sum_{i=1}^k \Big( \sum_{e \in \Pi_i} \weight(e) \Big)^{-1}.
    \end{equation}
\end{lemma}

\begin{proof}[Proof of Lemma~\ref{L:high_overlap}]
    Fix an edge $e =\{u,v\}$. Define $p_e = p_e(\omega) = \omega_e$ and $p^-_e = p^-_e(\omega) = \omega_e - e^{-n}$, and consider the event $\mathcal{G}_e$ that the only edge added in the graph process $G_{n,p}$ (coupled to $\omega$ as in Section~\ref{SS:coupleRG}) between $p^-_e$ and $p_e$ is the edge $e$, i.e.\
	\begin{equation} \label{eq:defG_e}
		\mathcal{G}_e := \big\{ G_{n, p_e} = G_{n, p^-_e} \cup \{ e \} \big\}.
	\end{equation}
	It may be helpful to think of $p^-_e$ as the (random) time just before the edge $e$ is added to the random graph. Using the distribution of the minimum difference of a collection of uniform random variables (see e.g. \cite[equation (6.2)]{MSS23}), one can easily show that $\mathcal{G}_e$ occurs with $\P$-probability at least $1 - O(n^{-4})$.

	Let $\epsilon = \epsilon(n) \geq n^{-1/10}$ be a sequence to be determined later (which will satisfy $\epsilon = o_\beta(1)$), and consider the following two events
	\begin{align}
		% \mathcal{A}_e = \mathcal{A}_e(\epsilon) &:= \{ \omega_e \not\in [1 - \frac{1 - \epsilon}{n}, 1 - \frac{1 + \epsilon}{n}]\} \\
		\mathcal{M}_e &:=  \big\{ e \in \textrm{MST} \big\} \cap \big\{ \omega_e \leq \frac{5 \log n}{n} \big\}, \label{eq:defM_E}\\
		\mathcal{S}_e = \mathcal{S}_e(\epsilon) &:= \big\{ |\mathcal{C}_2(p^-_e)| \leq L \epsilon^{-2} \log n \big\}, \label{eq:defC_e}
	\end{align}
	for some large enough but fixed constant $L$ as in Theorem~\ref{T:C2_size}.
    By Kruskal's algorithm, the event that $G_{n,p}$ is connected occurs if and only if every edge $f$ in the MST satisfies $\omega_f \leq p$. Thus, using Theorem~\ref{T:Gnp_connect}, we obtain%R
	\begin{align*}
		\P( \mathcal{M}_e) &\geq \P\big( \{e \in \textrm{MST} \} \cap \{ \forall f \in \textrm{MST}, \omega_f \leq 5\log n/n \big) \\
        & \geq \P(e \in \textrm{MST}) - \P(G_{n,5\log n/n} \text{ is not connected}) = \frac{2}{n} - O(n^{-4}).
	\end{align*}
	% Let 
	% \begin{equation*}
	% 	E^{*} := \Big\{ e \in E \, : \, \omega_e \not \in \Big[\frac{1 - \epsilon}{n}, \frac{1 + \epsilon}{n} \Big] \Big\},
	% \end{equation*}
        Consider now a sequence $0 = q_1 < q_2 < \ldots < q_k < q_{k+1}, \ldots < q_\ell = 1$ such that $q_k = (1-\epsilon)/n$, $q_{k+1} = (1+\epsilon)/n$, and $q_{i+1} - q_i \leq e^{-n}$ for $i \neq k$. Suppose that $p_e \in (q_{i}, q_{i+1})$ for some $i \neq k$, then on the event $\mathcal{G}_e$ we must have that $G_{n,p^-_e} = G_{n,q_i}$. Therefore,
	\begin{align*}
		\P( \mathcal{S}_e) &\geq \sum_{i=1, i \neq k}^{\ell-1}  \P \big( \mathcal{S}_e \cap \mathcal{G}_e \mid p_e \in (q_i, q_{i+1}) \big) (q_{i+1} - q_i)  \\
        &\geq \sum_{i=1, i \neq k}^{\ell-1} \Big( \P \big( |\cC_2(q_i)| \leq L \epsilon^{-2} \log n \mid p_e \in (q_i, q_{i+1}) \big) - \P\big(\mathcal{G}_e^c \mid p_e \in (q_i, q_{i+1}) \big) \Big)(q_{i+1} - q_i). 
	\end{align*}%
	% As one can see in the previous inequality, the event $\mathcal{S}_e$ is almost equivalent to $\{ e \in E^{*} \}$ up to small probability errors. We will make use of this observation again later.
    Note that the graph $G_{n,q_i}$ conditioned on $p_e \in (q_i, q_{i+1})$ is distributed as $G_{n,q_i}$ conditioned on $p_e > q_i$, since in both cases the distribution coincides with that of $G_{n,q_i} \setminus \{e\}$.  For small $q_i$, say $q_i \leq n^{-9/10}$, we then have
    \begin{align*}
         \P \Big( |\cC_2(q_i)| \leq L \epsilon^{-2} \log n \mid p_e \in (q_i, q_{i+1}) \Big) &= \frac{\P \big( \{ |\cC_2(q_i)| \leq L \epsilon^{-2} \log n \} \cap \{ p_e > q_i\} \big)}{\P( p_e > q_i)} \\
         &\geq \P \big( |\cC_2(q_i)| \leq L \epsilon^{-2} \log n  \big)  - q_i.
    \end{align*}
    For large $q_i \geq n^{-9/10}$, it is easy to verify that the degree of $u$ in $G_{n,q_i}$ is at least $2$ with probability $1 - O(n^{-4})$, and by Theorem~\ref{T:Gnp_connect} the graph $G_{n-1,q_i}$ is connected with  probability at least $1 - O(n^{-4})$. With high enough probability, we therefore have $|\cC_2| = 0$ even after $e$ is removed.
    Hence, using bounds on the size of the 2nd largest component of Theorem~\ref{T:C2_size} gives that
    \begin{equation}
        \P( \mathcal{S}_e) \geq  \sum_{i=1, i \neq k}^{\ell-1} \big( 1 - O(n^{-\eta'}) - O(n^{-4}) \big) (q_{i+1} - q_i) = 1 - \frac{2 \epsilon}{n} - O(n^{-\eta'}), \label{eq:S_e_bound}
    \end{equation}
    where $\eta' = \min\{ \eta, 9/10\}$ and $\eta$ is the constant from Theorem~\ref{T:C2_size}.
	
    Let $S(u)$ and $S(v)$ be the connected components of the random graph process at $p^-_e$, and w.l.o.g. assume that $|S(u)| \leq |S(v)|$.
    We write $E(S(u), S(u)^c)$ for the set of edges between $S(u)$ and $S(u)^c$. Observe that since $S(u)$ is not the largest component, we have
    \begin{equation*}
        |E(S(u), S(u)^c)| \leq n |S(u)| \leq n |\mathcal{C}_2(p^-_e)|.
    \end{equation*}
    Applying Kirchhoff's formula together with the Nash-Williams inequality using the (single) cutset $E(S(u), S(u)^c)$ gives that
	\begin{equation} \label{eq:nash_component}
		\bP(e \in \cT) \geq \frac{\weight(e)}{\sum_{f \in E(S(u), S(u)^c)} \weight(f)} = \frac{1}{1 + \sum_{f \in E(S(u), S(u)^c) \setminus \{e\}} e^{-\beta (\omega_f - \omega_e)} }.
	\end{equation}
	Notice that on $\mathcal{M}_e \cap \mathcal{G}_e$ the edge $e$ has the lowest weight among all edges in $E(S(u), S(u)^c)$. Furthermore, conditional on $\omega_e$ and $G_{n,p^-_e}$, the random variables in the collection
	\begin{equation*}
		\Big( \frac{\omega_f - \omega_e}{1- \omega_e} \Big)_{f \in E(S(u), S(u)^c) \setminus \{e \}}
	\end{equation*}
	are i.i.d.\ with law uniform on $[0,1]$, and, on the event $\mathcal{S}_e$, the cutset is bounded in size by
	\begin{equation*}
		\big| E(S(u), S(u)^c) \big| \leq n |\mathcal{C}_2(p^-_e)| \leq L n \epsilon^{-2} \log n.
	\end{equation*}
	Further, define 
	\begin{align}
		\mathcal{B}_e  &:= \Big\{\forall f \in E \big( S(u), S(u)^c \big) \setminus \{e \} \, : \, e^{-\beta(\omega_f - \omega_e)} < \frac{1}{n^5} \Big\} \nonumber \\
		&= \Big\{ \forall f \in E \big( S(u), S(u)^c \big) \setminus \{e \} \, : \, \frac{\omega_f-\omega_e}{1-\omega_e} >  5 \frac{\log n}{ \beta( 1-\omega_e)} \Big\} \label{eq:def_B_e}
	\end{align}
    Since on the event $\mathcal{M}_e \cap \mathcal{S}_e \cap \mathcal{G}_e$ we have $\omega_e \leq 5 \log n/n$,
    it follows that for $U$ distributed uniformly on $[0,1]$,  we have
	\begin{equation*}
        \P \big( \mathcal{B}_e \mid \mathcal{M}_e \cap \mathcal{S}_e \cap \mathcal{G}_e \big) 
        \geq \P \Big( U > 6 \frac{\log n}{ \beta} \Big)^{L n \epsilon^{-2} \log n} = \Big(1 - 6 \frac{\log n}{ \beta} \Big)^{L n \epsilon^{-2} \log n}. 
	\end{equation*}
	Hence, if for some $\gamma > 0$
	\begin{equation} \label{eq:condition_beta_eps}
		\frac{n \epsilon^{-2} (\log n)^2 }{\beta} \leq \gamma,
	\end{equation}
	then the event $\mathcal{B}_e$ occurs with probability at least 
    \begin{equation*}
        \Big(1 - 6 \frac{\log n}{ \beta} \Big)^{L \gamma \frac{\beta}{\log n}} \geq \exp \Big( - \frac{6 L \gamma}{1 -6 \log n/\beta } \Big) \geq 1 - 12 L \gamma,
    \end{equation*}%
    whenever $6 \log n/\beta < 1/2$.
	
	Notice that inequality \eqref{eq:nash_component} gives that restricted to the event $\mathcal{M}_e \cap \mathcal{S}_e  \cap \mathcal{G}_e \cap \mathcal{B}_e$, we have
	\begin{equation} \label{eq:e_in_cT_MST}
		\bP(e \in \cT) \geq \frac{1}{1 + n^{-3}} \geq 1 - n^{-3}.
	\end{equation}
	Hence, combing the previous bounds shows that for $\gamma \geq n^{-\eta'}$ %(where $\eta$ is the constant from Theorem~\ref{T:C2_size})
	\begin{equation*}
		\E \big[ \bP(e \in \cT)^2 \big] \geq (1 - n^{-3})^2 \P \big( \mathcal{B}_e \mid \mathcal{M}_e \cap \mathcal{S}_e \cap \mathcal{G}_e \big) \P\big( \mathcal{M}_e \cap \mathcal{S}_e \cap \mathcal{G}_e \big) \geq \big(1 - O(\gamma) \big) \big(1 - \epsilon \big) \frac{2}{n},
	\end{equation*}
	and in particular
	\begin{equation*}
		\E\big[ \mathcal{O}(\beta) \big] \geq \big(1 - O(\gamma) \big) \big(1 - \epsilon \big) n.
	\end{equation*}
	
	If $\beta \gg n (\log n)^2$, then we may pick $\gamma = o(1)$ and $\epsilon = o(1)$ such that the condition \eqref{eq:condition_beta_eps} is fulfilled. Furthermore, as trivially the second moment of $\mathcal{O}(\beta)$ is bounded from above by $n^2$, applying Chebyshev's inequality immediately gives \eqref{eq:high_overlap}.
\end{proof}

In the following, we briefly sketch how to obtain the equality \eqref{eq:length_high} in the high disorder regime of Theorem~\ref{T:length}. We also refer to \cite[Theorem~1.8]{K24}, where they only require that $\beta \gg n \log n$ instead of $\beta \gg n (\log n)^5$.

\begin{proof}[Proof sketch of Theorem~\ref{T:length}~\ref{T:length_high}]
	Denote by $M$ the MST coupled to the environment as in Section~\ref{SS:coupleRG}. As the MST has the minimum length given the environment $\omega$, we have that
	\begin{equation*}
		L(M) \leq L(\cT) \leq L(M) + |\cT \setminus M| \max_{e \in \cT} \omega_e,
	\end{equation*}
	where $|\cT \setminus M|$ is the number of edges in $\cT$ that are not in $M$. Let $\epsilon$ and $\gamma$ be as in \eqref{eq:condition_beta_eps}, then the above proof of Theorem~\ref{L:high_overlap} shows that the overlap between $M$ and $\cT$ satisfies
	\begin{equation*}
		\bE\big[ |M \cap \cT| \big] \geq \big(1 - O(\max \{ \epsilon, \gamma \}) \big) n,
	\end{equation*}
	with high probability. In particular,
	\begin{equation*}
		|\cT \setminus M| = O(\max \{ \epsilon, \gamma \}) n
	\end{equation*}
	with high probability. As $G_{n,p}$ is connected with high probability when $p = 2 \log n/n$, using Lemma~4.3 of \cite{MSS24}, one may easily obtain that $\max_{e \in \cT} \omega_e \leq 3 \log n/n$ with high probability. Equation \eqref{eq:length_high} now follows by letting $\epsilon, \gamma = o( (\log n)^{-1})$, which is possible provided that $\beta \gg n (\log n)^5$.
\end{proof}

\subsection{Local limit}

To show the local convergence towards a limit, we will proceed similarly as in the proof of Lemma~\ref{L:high_overlap}. However, instead of looking only at a fixed edge, we consider all outgoing edges of a vertex.
We first treat the case $r=1$ in the following lemma, and then extend this to the full local limit by using the union bound and transitivity. Recall that $B_\cT(u,r)$ and $B_M(u,r)$ are the closed balls of radius $r$ centred at $u$ in the RSTRE and MST, respectively.

\begin{lemma} \label{L:cT_M_deg}
	For any $u \in V$ we have that
	\begin{equation} \label{eq:cT_M_deg}
		\E \Big[ \bP \big( B_\cT(u,1 ) \neq B_M(u,1) \big) \Big]  = O\Bigg( \max\Bigg\{ \Big( \frac{n (\log n)^5}{\beta} \Big)^{1/3}, n^{-\eta'/2} \Bigg\}\Bigg),
	\end{equation}
	where $\eta' = \min\{ \eta, 1/10\}$ and $\eta$ is the constant from Theorem~\ref{T:C2_size}.
\end{lemma}

\begin{proof}
	We first show that the expectation of $\bP( B_M(u,1) \not\subseteq B_\cT(u,1 ))$ is small by restricting to the events as in the proof of Lemma~\ref{L:high_overlap}. Let
	\begin{equation*}
		\gamma = \epsilon = \max\Big\{ \Big( \frac{n (\log n)^2}{\beta} \Big)^{1/3}, n^{-2\eta'/3}\Big\}.
		%\epsilon &= \log n \gamma = \log n \Big( \frac{n}{\beta} \Big)^{1/3}
	\end{equation*}
	Our choice of $\epsilon$ and $\gamma$ guarantees that the condition in \eqref{eq:condition_beta_eps} holds (where we may implicitly assume that $\epsilon = \gamma = o(1)$) with $\epsilon \geq n^{-1/10}$ as required for Theorem~\ref{T:C2_size}. For $e = \{u,v\}$ recall the definitions of $\mathcal{G}_{\{u,v\}}$, $\mathcal{S}_{\{u,v\}}(\epsilon)$ and $\mathcal{B}_{\{u,v\}}$ from \eqref{eq:defG_e}, \eqref{eq:defC_e} and \eqref{eq:def_B_e}, % of Lemma~\ref{L:high_overlap}. 
	and further define
	\begin{align*}
		\mathcal{C} &:= \big\{ G_{n,5 \log n/n} \textrm{ is connected} \big\}, \\
		\mathcal{I} &:= \Big[0, \frac{1 - \epsilon}{n} \Big] \cup \Big[\frac{1+\epsilon}{n}, 5 \frac{\log n}{n}\Big], \\
		\mathcal{E} &:= \big\{ \text{for all } v \text{ such that } \{u,v\} \in M: \omega_{\{u,v\}} \in \mathcal{I} \big\}.
	\end{align*}
	
	By a union bound, we have that
	\begin{equation*}
		\bP\big( B_M(u,1) \not\subseteq B_\cT(u,1) \big) \leq \sum_{v \neq u} \bP\big( \{u,v\} \not\in \cT \big) 1_{\{u,v\} \in M}.
	\end{equation*}
	The arguments leading to \eqref{eq:e_in_cT_MST} show that, on the event
    \begin{equation*}
        \{ \{u,v\} \in M\} \cap \mathcal{C} \cap \mathcal{S}_{\{u,v\}} \cap \mathcal{G}_{\{u,v\}} \cap \mathcal{B}_{\{u,v\}}, 
    \end{equation*}
    we have $\bP( \{u,v\} \not\in \cT) \leq n^{-3}$, so that
	\begin{equation} \label{eq:uv_not_T_indicators}
		\sum_{v \neq u} \bP\big( \{u,v\} \not\in \cT \big) 1_{\{u,v\} \in M} 1_{\mathcal{G}_{\{u,v\}}} 1_{\mathcal{S}_{\{u,v\}}} 1_{\mathcal{B}_{\{u,v\}}}  1_{\mathcal{C}} \leq n^{-2}.
	\end{equation}
	We now split the expectation of $\bP( B_M(u,1) \not\subseteq B_\cT(u,1))$ into different terms depending on whether the events of the indicators in \eqref{eq:uv_not_T_indicators} hold or not. First, restrict to the event $\mathcal{G}_{\{u,v\}} \cap \mathcal{C} \cap \mathcal{E}$ such that with the bound around \eqref{eq:defG_e} and Theorem~\ref{T:Gnp_connect}, we have
	% \begin{equation*}
		%     \E \Big[ \bP\big( B_M(u,1) \not\subseteq B_\cT(u,1) \big)  \Big] &\leq \E \Big[ \bP\big( B_M(u,1) \not\subseteq B_\cT(u,1) \big) \mid \mathcal{C}, \mathcal{E} \Big] + \P( \mathcal{E}^c \cap \mathcal{C}) + O(n^{-3}) \\
		%     &\leq \sum_{v \neq u} \E \Big[  \bP\big( \{u,v\} \not\in \cT \big) \mid \{u,v\} \in M, \mathcal{G}_{\{u,v\}}, \mathcal{C},  \mathcal{E} \Big] + O(\epsilon),
		%     %&\qquad \qquad + \P( \mathcal{E}^c \cap \mathcal{C}) + O(n^{-3}),
		% \end{equation*}
	\begin{align*}
		\E \Big[ \bP\big( B_M(u,1) \not\subseteq B_\cT(u,1) \big)  \Big] &\leq \E \Big[ \bP\big( B_M(u,1) \not\subseteq B_\cT(u,1) \big) 1_{\mathcal{C}} 1_{\mathcal{E}} \Big] + \P( \mathcal{E}^c \cap \mathcal{C}) + O(n^{-4}) \\
		&\leq \E \Big[ \sum_{v \neq u} \bP\big( \{u,v\} \not\in \cT \big) 1_{\{u,v\} \in M} 1_{\mathcal{G}_{\{u,v\}}} 1_{\mathcal{C}} 1_{\mathcal{E}}  \Big] + O(\epsilon),
		%&\qquad \qquad + \P( \mathcal{E}^c \cap \mathcal{C}) + O(n^{-3}),
	\end{align*}
	where, using the coupling of the MST to the random graph from Section~\ref{SS:MST} (recall that the connected components of $M \cap G_{n,p}$ agree with those of $G_{n,p}$), we may bound
	\begin{equation*}
		\P( \mathcal{E}^c \cap \mathcal{C}) \leq \P \Big( \exists v \neq u : \omega_{\{u,v\}} \in \Big[ \frac{1 - \epsilon}{n}, \frac{1+\epsilon}{n} \Big] \Big)  + O(n^{-4})= O(\epsilon).
	\end{equation*}
	Restricting further to $\mathcal{S}_{\{u,v\}}$, we obtain
	\begin{align*}
		\E \Big[ \sum_{v \neq u} \bP\big( \{u,v\} \not\in \cT \big) 1_{\{u,v\} \in M} 1_{\mathcal{G}_{\{u,v\}}} 1_{\mathcal{C}} 1_{\mathcal{E}}  \Big] &\leq 
		\E \Big[ \sum_{v \neq u} \bP\big( \{u,v\} \not\in \cT \big) 1_{\{u,v\} \in M} 1_{\mathcal{G}_{\{u,v\}}} 1_{\mathcal{S}_{\{u,v\}}} 1_{\mathcal{C}}  \Big] \\
		&\qquad + \sum_{v \neq u} \P \Big(  \{ \{u,v\} \in M \} \cap \mathcal{S}_{\{u,v\}}^c \cap \mathcal{E} \Big).
	\end{align*}
	The bound from \eqref{eq:S_e_bound} (see also Theorem~\ref{T:C2_size}) gives us that
	\begin{equation*}
		\sum_{v \neq u} \P \Big( \{ \{u,v\} \in M \} \cap \mathcal{S}_{\{u,v\}}^c \cap \mathcal{E} \Big) \leq  n \P \big(\mathcal{S}_{\{u,v\}}^c \mid  \omega_{\{u,v\}} \in I \big) \P \big( \omega_{\{u,v\}} \in I \big) = O( n^{-\eta'} \log n).
	\end{equation*}
	Finally, restricting to $\mathcal{B}_{\{u,v\}}$ gives us with the condition \eqref{eq:condition_beta_eps} that
	\begin{align}
		\E \Big[ \bP\big( B_M(u,1) \not\subseteq B_\cT(u,1) \big)  \Big] &\leq \E \Big[ \sum_{v \neq u} \bP\big( \{u,v\} \not\in \cT \big) 1_{\{u,v\} \in M} 1_{\mathcal{G}_{\{u,v\}}} 1_{\mathcal{S}_{\{u,v\}}} 1_{\mathcal{B}_{\{u,v\}}}1_{\mathcal{C}}  \Big] \nonumber \\
		&+ \sum_{v \neq u} \P \Big( \mathcal{B}_e^c \mid \{ \{u,v\} \in M \} \cap \mathcal{S}_{\{u,v\}} \cap \mathcal{G}_{\{u,v\}} \cap \mathcal{C} \Big) \P \big( \{u,v\} \in M \big) + O(\epsilon) \nonumber \\
		&\leq n^{-2} +  n O(\gamma) \cdot \frac{2}{n}  + O(\epsilon) = O(\gamma). \label{eq:M_sub_T}
	\end{align}
	
	Consider now the degrees $\deg_M(u) = |B_M(u,1)|$ and $\deg_\cT(u) = |B_\cT(u,1)|$ of $u$ in $M$ and $\cT$, respectively. Both degrees have expectation $2 (n-1)/n$, and by the calculations above, the degree $\deg_\cT(u)$ almost stochastically dominates $\deg_M(u)$. We will use this fact now to show that the ball in $\cT$ is with high probability not larger than the ball in $M$. 
	
	By \eqref{eq:M_sub_T} we have
	\begin{align}
		\E \Big[ \bP \big( B_\cT(u,1) \not\subseteq B_M(u,1) \big) \Big] &\leq \E \Big[ \bP\big( \{ B_\cT(u,1) \not\subseteq B_M(u,1) \} \cap \{ B_M(u,1) \subseteq B_\cT(u,1)\}  \big) \Big] \nonumber\\
		&\qquad + \E \Big[ \bP \big( B_M(u,1) \not\subseteq B_\cT(u,1) \big) \Big] \nonumber \\
		&\leq \E \Big[ \bP \big( \deg_\cT(u) > \deg_M(u) \big) \Big] + O(\gamma). \label{eq:B_T_notsubset}  
	\end{align}
	Notice that \eqref{eq:M_sub_T} immediately also shows that
	\begin{equation*}
		\E \Big[ \bP \big(\deg_\cT(u) < \deg_M(u) \big) \Big] = O( \gamma ).
	\end{equation*}
	Therefore, considering the cases where $\deg_\cT(u)$ is larger than $\deg_M(u)$ or not, gives
	\begin{align}
		0 &= \E \Big[ \bE \big[ \deg_\cT(u) - \deg_M(u)\big] \Big] \nonumber \\
		&= \E \Big[ \bE\big[ \big(\deg_\cT(u) - \deg_M(u)\big) 1_{ \deg_\cT(u) > \deg_M(u)} \big] \Big] \nonumber \\
		&\qquad \qquad + \E \Big[ \bE \big[ \big(\deg_\cT(u) - \deg_M(u) \big) 1_{ \deg_\cT(u) < \deg_M(u)}\big] \Big]  \nonumber \\
		&\geq \E \big[ \bP \big( \deg_\cT(u) > \deg_M(u) \big) \big] - 60 \log n \E \big[ \bP \big( \deg_\cT(u) < \deg_M(u) \big) \big] \nonumber \\
		&\qquad \qquad - n \P \big( \deg_M(u) > 60 \log n \big).  \label{eq:deg_split_case}
	\end{align}
	Rearranging \eqref{eq:deg_split_case} and using Theorem~\ref{T:Gnp_connect} shows that 
	\begin{equation}
		\E \big[ \bP \big( \deg_\cT(u) > \deg_M(u) \big) \big] = O( \gamma \log n ). \label{eq:degT>degM}
	\end{equation}
	Putting \eqref{eq:M_sub_T}, \eqref{eq:B_T_notsubset} and \eqref{eq:degT>degM} together gives that
    \begin{align*}
        \E \Big[ \bP \big( B_\cT(u,1) \neq B_M(u,1) \big) \Big] &\leq \E \Big[ \bP \big( B_\cT(u,1) \not\subseteq B_M(u,1) \big) \Big] + \E \Big[ \bP \big( B_M(u,1) \not\subseteq B_\cT(u,1) \big) \Big] \\
        &\leq \E \Big[ \bP \big( \deg_\cT(u) > \deg_M(u) \big) \Big] + O(\gamma) = O(\gamma \log n),
    \end{align*}
    which finishes the proof by our choice of $\gamma$.
\end{proof}

As a corollary of Lemma~\ref{L:cT_M_deg}, we may now obtain the local limit of the RSTRE for $\beta = n (\log n)^{\lambda(n)}$ for diverging $\lambda(n)$.

\begin{proof}[Proof of Theorem~\ref{T:local_limit}~\ref{T:local_limit_high}]
	Fix $r \geq 1$. If $B_M(1,r) \neq B_\cT(1,r)$, then there must exist at least one vertex $u$ in $B_M(1,r)$ such that $B_M(u,1) \neq B_\cT(u,1)$. A union bound gives that
	\begin{align}
		\E \Big[ \bP \big( B_\cT(1,r  ) \neq B_M(1,r) \big) \Big] &\leq
		\E \Big[ \bP \big( B_\cT(1,1) \neq B_M(1,1) \big) \Big] \nonumber  \\
		&\qquad + \sum_{u \neq 1} \E \Big[ \bP \big( B_\cT(u,1) \neq B_M(u,1) \big) 1_{u \in B_M(1,r)}\Big]. \label{eq:B_neq_M_union}
	\end{align}
	Using exchangeability of the edge weights (due to transitivity of the complete graph) and the fact that $u \in B_M(1,r)$ if and only if $1 \in B_M(u,r)$, gives for any $v \neq u$ that
	\begin{align*}
		\E \Big[ \bP \big( B_\cT(u,1) \neq B_M(u,1) \big) 1_{u \in B_M(1,r)}\Big] &= \E \Big[ \bP \big( B_\cT(u,1) \neq B_M(u,1) \big) 1_{ 1 \in B_M(u,r)}\Big] \\
		&= \E \Big[ \bP \big( B_\cT(u,1) \neq B_M(u,1) \big) 1_{ v \in B_M(u,r)}\Big].
	\end{align*}
	It follows that
	\begin{align}
		&\E \Big[ \bP \big( B_\cT(u,1) \neq B_M(u,1) \big) 1_{u \in B_M(1,r)}\Big] \nonumber \\
		&\qquad \qquad = \frac{1}{n-1} \sum_{v \neq u} \E \Big[ \bP \big( B_\cT(u,1) \neq B_M(u,1) \big) 1_{v \in B_M(u,r)}\Big] \nonumber \\
		&\qquad \qquad = \frac{1}{n-1} \E \Big[ \bP \big( B_\cT(u,1) \neq B_M(u,1) \big) \big(\big|B_M(u,r)\big|-1 \big)\Big] \label{eq:B_neq_M_size_B}
	\end{align}
	In view of Lemma~\ref{L:max_deg_MST}, the size of $|B_M(u,r)|$ is bounded with high probability by $60^r (\log n)^r$, so that \eqref{eq:B_neq_M_size_B} can be bounded by
	\begin{equation} \label{eq:B_M_size_bound}
		\frac{60^r (\log n)^r}{n-1}   \E \Big[ \bP \big( B_\cT(u,1) \neq B_M(u,1) \big) \Big]  + O(n^{-4}).
	\end{equation}
	Plugging this bound into \eqref{eq:B_neq_M_union}, and using that $\beta \gg n (\log n)^{5 + 3r}$ shows with Lemma~\ref{L:cT_M_deg} that
	\begin{equation*}
		\E \Big[ \bP \big( B_\cT(1,r  ) \neq B_M(1,r) \big) \Big] = o_\beta(1).
	\end{equation*}
	
	Consequently, for any finite rooted tree $t$ of height $r$
	\begin{align*}
		\widehat{\P}( B_\cT(1,r) \simeq t) %&\leq \P( B_M(1,r) \simeq t) + \E \Big[ \bP \big( B_\cT(1,r) \simeq t) 1_{B_M(1,r) \not\simeq t} \Big] \\
		&\leq \P(B_M(1,r) \simeq t) + \E \Big[ \bP \big( B_\cT(1,r) \neq B_M(1,r) \big) \Big] \\
		&= \P( B_M(1,r) \simeq t) + o_\beta(1)
	\end{align*}
	and
	\begin{align*}
		\widehat{\P}( B_\cT(1,r) \simeq t) %&\geq \E \Big[ \bP \big( B_\cT(1,r) \simeq t) 1_{B_M(1,r) \simeq t} \Big] \\
		&\geq \P( B_M(1,r) \simeq t) - \E \Big[ \bP \big( B_\cT(1,r) \neq B_M(1,r) \big) \Big] \\
		&= \P( B_M(1,r) \simeq t) - o_\beta(1).
	\end{align*}
	The local convergence of the MST on the complete graph in Theorem~\ref{T:local_MST} then also implies the local convergence of the RSTRE (to the same limiting object).
\end{proof}

\begin{remark} \label{R:MST_growth}
	The argument after \eqref{eq:B_M_size_bound} also shows that the neighborhoods $B_\cT(u,1)$ and $B_M(u,1)$ agree with each other for diverging $r$, as long as $\beta$ is large enough such that
	\begin{equation*}
		60^r (\log n)^r  \Big( \frac{n (\log n)^5}{\beta} \Big)^{1/3} \longrightarrow 0.
	\end{equation*}
	% We expect that this condition can be somewhat relaxed. Indeed, for \eqref{eq:B_M_size_bound} we used that, with high probability, the neighborhood $B_M(1,r)$ grows in size at most as fast as $60^r (\log n)^r$, however, as shown in \cite{Add13}, the typical size of a neighborhood is of order $r^{3}$. We are not aware of any better tail bounds, say $|B_M(1,r)| = O( r^k \log n)$ for some $k \geq 1$, that hold with large enough probability to make \eqref{eq:B_M_size_bound} work for smaller $\beta$.
	We expect that this condition can be somewhat relaxed. Indeed, for \eqref{eq:B_M_size_bound}, we used the fact that, with high probability, the neighborhood $B_M(1,r)$ grows in size at most as fast as $60^r (\log n)^r$. However, as shown in \cite{Add13}, the typical size of a neighborhood is of order $r^3$. We are not aware of any better tail bounds, such as $|B_M(1,r)| = O(r^k \log n)$ for some $k \geq 1$, that hold with sufficiently high probability to make \eqref{eq:B_M_size_bound} work for smaller $\beta$.
\end{remark}

\section*{Acknowledgements}
The author is grateful to Rongfeng Sun for helpful comments on a draft of this paper. Furthermore, the author would like to thank Michele Salvi for raising interesting questions about the RSTRE, and an anonymous referee for their many constructive remarks.

\vspace*{-0.1cm}

\bibliographystyle{plain}
\bibliography{RSTRE}

\end{document}